\documentclass[a4paper,10pt,reqno]{article}
\usepackage[utf8]{inputenc}
\usepackage[english]{babel}
\usepackage[margin=1.1in]{geometry} 
\usepackage{amsmath, amsthm, amssymb}
\usepackage{verbatim}
\usepackage{graphicx}
\usepackage{caption}
\usepackage{subcaption}
\usepackage{float}
\usepackage{tikz-cd}
\usepackage{tikz}
\usepackage{setspace}
\usepackage{mathrsfs}
\usepackage{csquotes}
\usepackage{bbm}
\usepackage{accents}
\usepackage{enumitem}
\usepackage{mathtools}
\usepackage{imakeidx}
\usepackage[T1]{fontenc}
\usepackage{mdframed}
\usepackage{todonotes}
\usepackage[final]{hyperref} 
\usepackage{nicefrac}
\usepackage{authblk}

\hypersetup{
	colorlinks=true,       
	linkcolor=blue,        
	citecolor=blue,        
	filecolor=magenta,     
	urlcolor=blue         
}
\usepackage{thmtools} 
\usepackage[nameinlink,capitalise]{cleveref}
\allowdisplaybreaks

\addto\extrasenglish{

}

\newcommand{\B}{\mathcal{B}}

\newcommand{\E}{\mathcal{E}}
\newcommand{\EE}{\mathbb{E}}
\newcommand{\F}{\mathcal{F}}

\renewcommand{\H}{\mathcal{H}}

\newcommand{\R}{\mathbb{R}}

\newcommand{\T}{\mathbb{T}}

\newcommand{\Z}{\mathbb{Z}}

\newcommand{\mcA}{\mathcal{A}}

\newcommand{\mcD}{\mathcal{D}}
\newcommand{\mcE}{\mathcal{E}}
\newcommand{\mcF}{\mathcal{F}}

\newcommand{\mcH}{\mathcal{H}}

\newcommand{\mcL}{\mathcal{L}}

\newcommand{\mcV}{\mathcal{V}}

\newcommand{\mbE}{\mathbb{E}}
\newcommand{\mbF}{\mathbb{F}}

\newcommand{\mbI}{\mathbb{I}}

\newcommand{\mbN}{\mathbb{N}}

\newcommand{\mbP}{\mathbb{P}}

\newcommand{\mbR}{\mathbb{R}}
\newcommand{\mbS}{\mathbb{S}}
\newcommand{\mbT}{\mathbb{T}}

\newcommand{\mbZ}{\mathbb{Z}}

\DeclareMathOperator*{\argmin}{arg\,min}
\DeclareMathOperator{\grad}{\nabla}

\DeclareMathOperator*{\esssup}{ess\,sup}

\renewcommand{\epsilon}{\varepsilon}
\renewcommand{\setminus}{\smallsetminus}
\newcommand{\eps}{\epsilon}

\newcommand{\brak}[1]{\left\langle#1\right\rangle}

\newcommand{\expt}[2][]{\mathbb{E}_{#1}\left[#2\right]}

\newcommand{\RN}[1]{%
	\textup{\uppercase\expandafter{\romannumeral#1}}%
}

\newtheorem{theorem}{Theorem}[section]

\newtheorem{definition}[theorem]{Definition}

\newtheorem{corollary}[theorem]{Corollary}
\newtheorem{lemma}[theorem]{Lemma}
\newtheorem{proposition}[theorem]{Proposition}

\newcommand{\supp}{\text{supp}}
\newcommand{\dd}{\mathop{}\!\mathrm{d}}

\theoremstyle{remark}
\newtheorem{remark}[theorem]{Remark}

\numberwithin{equation}{section}

\long\def\avi#1{{\color{red}A:\ #1}}

\long\def\silvia#1{{\color{teal}S:\ #1}}

\title{Homogenisation of a Passive Scalar Transported by Locally Supported White Noise}
\author[1]{Federico Butori}
\affil[1]{federico.butori@sns.it \protect\\ Scuola Normale Superiore, Classe di Scienze \protect\\ Piazza dei Cavalieri, 7, 56126 Pisa, Italia
}

\author[2]{Avi Mayorcas}
\affil[2]{ am2735@bath.ac.uk \protect\\ Department of Mathematical Sciences, University of Bath \protect\\ North Rd, Claverton Down, Bath BA2 7AY
}

\author[3]{Silvia Morlacchi}
\affil[3]{silvia.morlacchi@dm.unipi.it \protect\\ Dipartimento di Matematica, Universit\`a di Pisa \protect\\ Largo Pontecorvo 5, 56127 Pisa, Italia
}

\date{\today}

\begin{document}
\maketitle

\begin{abstract}
%
Stochastic perturbations of transport type are a common and widely accepted way of representing turbulent effects in fluid dynamics models. In many known examples, it even leads to improved solution theory, a phenomenon known as \emph{regularization by noise}. A common thread in the recent literature on the topic is the so-called \emph{It\^o-Stratonovich diffusion limit}. By selecting Stratonovich transport noise with carefully arranged vector fields, one can show that the solution of certain SPDEs are close, in an appropriate topology, to an effective, deterministic, equation with a new effective second order elliptic operator, linked to the Ito-Stratonovich corrector.
In this work, we deal with a passive scalar model with molecular diffusivity $\kappa$. Starting from the results in [Flandoli \emph{et al.}, 2022, \emph{Philos. Trans. Roy. Soc. A}, 380(2219)], we consider a transport noise made by a sum of independent and compactly supported vector fields. 
This setting is relevant for models of stratified turbulence which naturally occur in boundary layers and Boussinesq models. 
Due to the anisotropic nature of the noise, the identification of the limit equation is not straightforward as in all other examples known in literature, as the Ito-Stratonovich corrector is a generic second order elliptic operator with non-constant coefficients.
Using tools from Homogenisation theory, we obtain a representation for the limiting effective diffusivity matrix.
Exploiting this representation, we study asymptotics, in the $\kappa \rightarrow 0$ regime, of the effective diffusivity across a number of vector field regimes parametrised by the radius of their support. Finally, we provide a careful numerical analysis of the effective diffusivity, discovering a nonlinear behavior for $\kappa \rightarrow 0$, in some regimes. 

\vspace{0.3cm}

\begin{small}
\end{small}


\end{abstract}


\tableofcontents

\section{Introduction}
We study scaling limits, as $N\to +\infty$, of solutions to the following stochastic passive scalar equation on $\mbR_+ \times \mbT^2$,
\begin{equation}\label{eq:main_strat_N}
    d u^N_t = \sqrt{2}\sum_{k \in \mbZ^2} \sigma_{k}^N\cdot \grad u^N_t \circ dW_t^{k} + \kappa\Delta u^N_t dt,
\end{equation}
where $\sigma^N_k$ are radial divergence free vector fields, centred at lattice points $\nicefrac{k}{N}$, $k \in \Z^2$ and support contained within balls of radius proportional to $N^{-1}$ (see \autoref{sec:definition_of_noise}, Assumption \ref{ass 1}). We refer to $\sigma_k^N$ as vorticity \emph{patches}. Finally, we introduce a parameter $c>0$ which governs the overlap of adjacent patches: for small values of $c$ the patches are sparse, with little to no overlap, while for larger values they cover the space domain multiple times.
We are especially interested in the behaviour of the limiting equation ($N\to \infty$) in the small molecular diffusivity ($\kappa\ll 1$) and sparse ($c \le \nicefrac{\sqrt{5}}{2}$) regime. The treatment of the sparse regime, in particular, is one of the main novelties of this work (\emph{c.f} \autoref{rem:unif_ell}), which we can treat thanks to some tools coming from homogenization theory, never employed before in this context. 
The precise set-up of the noise and vector fields are given below.

The SPDE \eqref{eq:main_strat_N} is a simple model of a passive scalar diffusing at rate $\kappa>0$ in a stochastic environment which is spatially varying, Gaussian and white in time. The stochastic transport is motivated on the one hand as a simplified model of a turbulent fluid (see \cite{kraichnan_68,kraichnan_94_anomalous,chertkov_falkovich_96_anomalous,frisch_mazzino_vergassola_98_intermittency,majda_kramer_99_turbulence,chaves_gawcedzki_horvai_kupianen_vergassola_03_lagrangian,sreenivasan_19_turbulent}, and also \cite{flandoli2022additive, debussche_pappalettera_24_second_order}) and on the other from the perspective of regularisation by noise \cite{majda_kramer_99_turbulence,flandoli_gubinelli_priola_10_wellposed,brzezniak_existence_nodate, galeati_leahy_nilssen_25_rough, roveri_well-posedness_2024, flandoli_galeati_luo_21_delayed,flandoli_luo_21_highmode, coghi_maurelli_24_kraichnan, galeati_grotto_maurelli_24_anomalous, agresti_global_2024}. Our particular choice of noise is motivated, in part, by models of eddy turbulence generated by a fluid moving in a bounded domain, \cite{flandoli_galeati_luo_22_eddy,flandoli_luongo_22_channel}. While we treat \eqref{eq:main_strat_N} on the torus, the locally supported vector fields are analogous to those considered in \cite{flandoli_galeati_luo_22_eddy} which allowed the authors to represent boundary layer turbulence which forms when a viscous fluids moves through a bounded domain. The model considered in \cite{flandoli_galeati_luo_22_eddy} is  variant of the Boussinesq model for heat conductance in a bounded domain where the vector fields $\{\sigma_k\}_{k\geq 0}$ vanish at the boundary, in order to model a no-slip boundary condition on the turbulent fluid and $u^N$ represents a scalar field transported by the noise and diffusing within the domain. In that work the authors show that for an appropriate choice of vector fields and $N$ sufficiently large, solutions to \eqref{eq:main_strat_N} are arbitrarily close, in mean square, to solutions of an \emph{enhanced} divergence form parabolic equation, depending on $N$. A similar analysis is given, for example, in \cite[Sec.~4.1]{majda_kramer_99_turbulence}.

Scaling limits of stochastic transport equations were first considered in \cite{galeati_20_convergence}, leading to a large number of works (of which we only mention a selection) \cite{flandoli_scaling_2021, flandoli_galeati_luo_21_delayed, flandoli_galeati_luo_21_mixing,flandoli_luo_21_highmode, flandoli_2d_2023, butori_luongo_24_magnetic, butori_mean-field_2025, agresti2024AnomalousDissipationInduced, butori_background_2024}, considering scaling limits of linear and non-linear equations similar to \eqref{eq:main_strat_N}. The backbone of thess works lies in rewriting the Stratonovich SPDE in its (formally equivalent) It\^o form, introducing a corrector term, which usually takes the form of a dissipative second order operator, and designing a sequence of transport coefficients, with high spatial frequency, such that the martingale term in the equation becomes infinitesimal in negative topologies, while the corrector term stays of order one (we give more a technical discussion in \autoref{subsec: discussion}). \\
However, all the results mentioned above are concerned with transport noise defined such that the Stratonovich corrector can directly be seen to be strictly dissipative, most commonly (with the only exceptions of \cite{flandoli_galeati_luo_22_eddy, flandoli_luongo_22_channel}) being a constant multiple of the Laplacian. By considering anisotropic vector fields (i.e. with local support), the nice form of the Stratonovich corrector is lost, and it becomes a general second order elliptic operator with non-constant coefficients, depending on the scaling parameter $N$. This leads us to apply tools of homogenization theory to pass to a unified scaling limit in both the martingale noise and second order operator. 

More specifically, the scaling argument shows that, for $N$ large enough, solutions to the stochastic equation \eqref{eq:main_strat_N} are close, in an appropriate weak sense, to solutions of the parabolic PDE with \textit{modified} diffusivity 
\begin{equation*}
    \partial_t \tilde u^N = \operatorname{div}((\kappa I + A^N(x)) \nabla \tilde u^N_t).
\end{equation*}
The matrix $A^N$ is positive semi-definite everywhere and is determined by the choice of vector fields $\{\sigma_k\}_{k\geq 0}$. However, in the case we are mostly interested, where the vector fields are sparse ($c\le \sqrt{5}/2$), $A^N$ is not uniformly strictly positive definite, thus making unclear whether the second order operator appearing in the PDE is `more elliptic' than the original laplacian (i.e. the \emph{effective diffusivity} is larger than $\kappa$). We refer again the reader to the discussion in \autoref{subsec: discussion}. 
\\

The first goal of this work is to complete the scaling argument in \cite{flandoli_galeati_luo_22_eddy}, by identifying the limit parabolic equation for $N \rightarrow \infty$; this is the content of \autoref{th:main_introduction}. We do so employing, for the first time in this context, tools from Homogenization theory, which readily gives us a way of quantifying the turbulent effective diffusivity of the limiting model. To the best of our knowledge, even if heuristically the \emph{It\^o-Stratonovich diffusion limit} has always been interpreted as a sort of homogenization technique, this is the first time in which this parallel is rigorously made explicit. 
Moreover, this allows us to consider more general choices of the vector fields $\{\sigma_k\}_{k\ge 0}$ (i.g. the sparse $c\le \sqrt{5}/2$ case) with respect to the existing literature, even if not all of them will be compatible with a form of enhancement of dissipation. 
%
This leads to our second and main motivation for this work, to investigate the behaviour of the effective diffusivity of the homogenized model, with respect to our main parameters, the radius of the patches $c$ and the initial molecular diffusivity $\kappa$; this is the content of \autoref{thm: effective_bounds_intro}. This analysis is not straightforward in the sparse case ($c< \sqrt{5}/2$), due to the lack of uniform ellipticity of the Statonovich corrector $\operatorname{div}(A^N \nabla \cdot)$ (a property always assumed in all the above mentioned works, see \autoref{rem:unif_ell})

The analysis is based on variational tools from homogenization theory, which we have access to from the first part. At the end of the work, we present also numerical results which complete the analysis presented in the second part. The numerical results aim to clarify the exact asymptotic behaviour of the effective diffusivity for low values of $\kappa$, and to show how, depending on $c$, there are clear transition between different regimes. 
\\

The manuscript is arranged as follows: in \autoref{sec: setting} we introduce the setting and we define the basics objects of our analysis; in \autoref{sec: scaling} we prove the main scaling limit, identifying the limit \emph{homogenized} PDE, proving \autoref{th:main_introduction}. In \autoref{sec:diffusivity} we introduce the variational setting for the Homogenization problem of the preceding section and use it to deduce bounds on the \emph{effective diffusivity matrix}, proving \autoref{thm: effective_bounds_intro}. Finally in \autoref{sec: numerics} we present some numerical computations to fill the gaps of our analysis and deduce precise asymptotics of the effective diffusivity for low values of $\kappa$, across several patches density regimes, depending on $c$. 

\subsection{Notation and Preliminaries}

\begin{itemize}
	\item $\mbN \coloneqq \{0,1,\ldots\}$, $\mbN_0 \coloneqq \mbN \setminus \{0\}$, $\mbZ = -\mbN \cup \mbN$ and $\mbZ_0 = \mbZ\setminus \{0\}$.
\item We set $\mbT^2 \coloneqq \mbR^2/\mbZ$ and without loss of generality identify $\mbT^2 =[0,1]^2$ with periodic boundary conditions.
    %
	%
	%
\item For $x = (x_1,x_2)\in \mbR^2$ we take the convention $x^\perp = (-x_2,x_1)$. So that $x^\perp$ is a counter-clockwise rotation of $x$ through $\nicefrac{\pi}{2}$.
\item We define the finite box of side length $n\in \mbN$ with bottom left corner $k\in \mbZ^2$ by the short-hand notation
\begin{equation*}
    \Box^2_n(k) \coloneqq \mbZ^2 \cap \left([k_1,n]\times [k_2,n]\right).
\end{equation*}
where it will not cause confusion we simply set $\Box^2_n \coloneqq \Box^2_n((0,0))$.
    
    \item We use the letter $\mathcal{L}$ to denote linear, possibly unbounded, operators between Banach space and we will indicate their domain with $\mathcal{D}(\mathcal{L})$. 
	%
	\item Given $n\geq 1$ and a measurable map $f:\mbT^2 \to \mbR^n$ we set
	\begin{equation*}
		\|f\|_{L^p_x} \coloneqq \begin{cases}
			\left(\int_{\mbT^2} |f(x)|^p\dd x\right)^{\nicefrac{1}{p}}, & p\in [1,+\infty),\\
			\esssup_{x\in \mbT^2} |f(x)|, & p=+\infty,
		\end{cases}
	\end{equation*}
	and
	\begin{equation*}
		L^p(\mbT^2;\mbR^n) \coloneqq \{ f:\mbT^2 \to \mbR^n \,:\, \|f\|_{L^p_x} <\infty \}, \quad L^p_0(\mbT^2;\mbR^n) \coloneqq \{ f\in L^p(\mbT^2;\mbR^n)\,:\,  \langle f,1\rangle_{L^2(\mbT^2)}=0\, \}.
	\end{equation*}
    \item For $s\in \R$ we indicate with $\H^s(\T^2; \R^n)$ the usual Sobolev spaces of periodic functions, moreover we will indicate with $\H^s_0(\T^2, \R^n)$ the subspace of $\H^s(\T^2, \R^n)$ made by functions of zero mean, endowed with the usual homogeneous norm. 
    %
\item For $n\in \mbN_0$ we let $C^\infty(\mbT^2;\mbR^n)$ denote the Fr\'echet space of smooth periodic functions and $C^\infty_0(\mbT^2;\mbR^n)$ denote those which are mean free i.e. $\phi \in C^\infty(\mbT^2;\mbR^n)$ such that $\langle \phi_i,1\rangle_{L^2(\mbT^2)} =0$ for all $i=1,\ldots,d$. 
\item By convention, when we do not specify the range in a function space we mean the space of appropriate scalar maps. For example we set $L^2(\mbT^2)= L^2(\mbT^2;\mbR)$ etc.
    \item Given two Banach space valued semi-martingales $X,\,Y:\Omega\times \mbR_+ \to E$ defined on the same filtration $\F_t$, we write $\mbR_+ \ni t\mapsto [X]_t$ for the quadratic variation and $\mbR_+ \ni t\mapsto [X,Y]_t$ for the covariation.
\item Given a separable Hilbert space $E$, a bounded interval $I\subseteq \mbR_+$ and a filtration $\mbF \coloneqq \{\mcF_t\}_{t\in I}$, we let $L^2_{\mbF}(I; E)$ denote the space of progressively measurable processes $Z : \Omega\times I\to X$ such that
\begin{equation*}
    \EE\Bigg[\int_0^T \|Z_t\|_{X}^2dt\Bigg] < +\infty
\end{equation*}
and $C_{\mbF}(I ; X)$ to be the subspace of continuous adapted processes such that 
$$\expt[]{\sup_{t\in [0,T]}\|Z_t\|^2_X} < \infty.$$
\item Given a $2\times 2$ matrix $M$, possibly depending on the space parameter $x\in \T^2$, we say that $M$ is $\lambda$-\emph{elliptic}, or simply \emph{elliptic} if for every unitary vector $\xi \in \mathbb{S}^1$ it holds 
$$\lambda \le \xi \cdot M\xi \le \frac{1}{\lambda} .$$
If $M$ is space dependent, we say that it is \emph{uniformly} elliptic, if it is $\lambda$-elliptic for every $x\in \T^2$ for some $\lambda$ independent of $x$. 
\end{itemize}
\subsection{Elements of Homogenization Theory} 
We briefly recall the main setting of the elliptic Homogenization Theory, which we shall employ later. 
Let $M(x)$ be a smooth field of symmetric, uniformly elliptic matrices on the periodic domain $\T^2$ and for any $N>1$ and $f^N\in \mcH^{-1}(\T^2)$ consider the elliptic problem, understood in its weak formulation 
\begin{equation}\label{abs_ell_intro}
    \nabla \cdot (M(Nx) \nabla u^N(x))=f^N(x), \qquad u^N \in \dot \H^1(\T^2).
    \end{equation}

Setting $e_1=(1,0)$, $e_2=(0,1)$ we let $\phi_i$ (for $i=1,2$), be the solution of the elliptic PDE problem, still understood in its weak formulation 
\begin{align}\label{eq:corrector_gen}
    \begin{cases}
        \nabla\cdot \Big(M(x) e_i+ M(x) \nabla \phi_i(x)\Big) =0 \\
        \phi_i \in \dot \mcH^1_{per}(\mbT^2)
    \end{cases}
\end{align}
which is well-posed, thanks to the uniform ellipticity of $M(x)$, (\emph{c.f.} \cite[Sec. 6.2.1]{evans_10}).

We define the matrix with constant entries $\bar{M}$ as 
\begin{equation} \label{def: hom_H_gen}
    \bar M_{ij} = (\bar M e_j)\cdot e_i := \int_{\mbT^2} (M(x)(e_j + \nabla \phi_j(x)))\cdot e_i\dd x = \int_{\mbT^2} \left(M_{ij}(x) + (M(x) \nabla\phi_j(x))\right)\cdot e_i \dd x
\end{equation}
Then, informally, the matrix $\bar M$, describes the effective diffusivity of \eqref{abs_ell_intro} in the limit $N\rightarrow \infty$. More precisely, provided that $f^N \rightarrow f$ in $\dot \H^{-1}(\T^2)$, then the solution $u^N$ to \eqref{abs_ell_intro} converges in $L^2(\T^2)$ to the unique solution $u(x) \in \dot \H^1(\T^2)$ of 
\begin{equation}\label{abs_ell_hom_intro}
    \nabla \cdot (\bar M \nabla u(x))=f(x)
    \end{equation}
 We refer the reader to \cite[Chapter 1]{bensoussan1978AsymptoticAnalysisPeriodic} for a precise statement and for a proof and further discussion. 
 In this manuscript we will mostly be concerned with the parabolic problem rather than the elliptic one, however since $M$ will always be time independent in our setting, the homogenization of the parabolic problem reduces to homogenization of the elliptic one (see \cite[Ch.~2, Rem.~1.6]{bensoussan1978AsymptoticAnalysisPeriodic}). Specifically, the solution of 
 \begin{equation}\label{abs_ell}
    \partial_t u^N(t,x)=\nabla \cdot (M(Nx) \nabla u^N(t,x)), \qquad u^N(0, \cdot )= u_0^N\in \dot \H^1(\T^2), \quad u^N \in C(0, T; \dot \H^1(\T^2)).
    \end{equation}
converges in $L^2(0, T; L^2(\T^2))$ to the unique solution $u(t,x) \in C(0, T; \dot \H^1(\T^2))$ of
\begin{equation}\label{abs_ell_intro_parab}
    \partial_t u(t,x)=\nabla \cdot (\bar M(x) \nabla u(t,x)), \qquad u(0, \cdot )= u_0\in \dot \H^1(\T^2), 
    \end{equation}
provided that $u^N_0 \rightarrow u_0$ in $L^2(\T^2)$
 In our setting, we give an analogous version of this statement in \autoref{prop:pde_homogenisation}.
 
 Finally, we will need the following properties of the homogenized matrix $\bar M$, whose proofs, for the reader's convenience, are recalled in \autoref{app_A1}.
 \begin{lemma}\label{lem:homogenised_upper_lower_symmetric}
    Let $M(x)$ be a smooth and periodic field of symmetric matrices and let $0\leq \lambda \leq \Lambda <+\infty$ be such that
    \begin{equation}
     \sup_{x\in \mbT^2}|M(x)\xi| \leq \Lambda  \xi \quad \text{and}\quad \inf_{x\in \mbT^2} M(x)\xi \cdot \xi \geq \lambda |\xi|^2 \quad \text{for all }\, \xi\in \mbR^2.  
    \end{equation}
    Furthermore, we fix $\phi_i\in \H^1(\T^2)$, for $i=1,2$, weak solutions to \eqref{eq:corrector_gen}. Then, the following facts  hold.
    \begin{enumerate}[label=\roman*)]
        \item For every $\xi\in \mbR^2$,
        \begin{equation}
            |\bar{M} \xi| \leq  \Lambda |\xi| \left(\sum_{i=1}^2 \int_{\mbT^2} \left|e_i + \nabla \phi_i(x)\right|^2\dd x\right)^{\frac{1}{2}}.
        \end{equation}
        \label{it:bar_H_kappa_upper_bnd}
        \item For every $\xi \in \mbR^2$,
        \begin{equation}
            \bar{M} \xi \cdot \xi \geq  \lambda |\xi|^2.
        \end{equation}
        \label{it:bar_H_kappa_lower_bnd}
        \item The matrix $\bar{M}$ is symmetric. \label{it:bar_H_kappa_symmetric}
    \end{enumerate}
\end{lemma}
\begin{proof}
    See \autoref{sec:standard_homogenised_statements}.
\end{proof}

\section{Setting}\label{sec: setting}
We deal with the stochastic passive scalar equation \eqref{eq:main_strat_N}, posed on $\mbR_+ \times \mbT^2$.
%
The noise coefficients $\sigma_k$ are spatially localized vorticity \textit{patches}, $\left\{W_t^k\right\}_{k\ge 0}$ is a sequence of i.i.d. standard Brownian motions and $\kappa>0$ is the \emph{molecular diffusivity}.
In this section we give the definition of the noise and of the main objects of our analysis. 
\subsection{Definition of the Noise}\label{sec:definition_of_noise}
We begin by constructing a
family of locally supported vortex \emph{patches}. Let $ \Lambda^N$ be the integer lattice of $\R^2$ with resolution $\nicefrac{1}{N}$ for some integer $N\in \mbN$. That is, 
\begin{equation}\label{eq:lattice_def_1}
     \Lambda^N \coloneqq \left(\frac{1}{N} \Z\times \frac{1}{N}\mbZ\right).
 \end{equation} 
We enumerate $\Lambda^1$ by $\mbZ^2$ and then set
\begin{equation}\label{eq:lattice_def_2}
     \Lambda^N \ni x^N_k \coloneqq \frac{k}{N},
    \quad k \in \mbZ^2 .
\end{equation}
 We will centre each vortex patch on the nodes of this lattice. The vorticity will be constructed by defining its stream function. Given a constant $c\in (0,+\infty)$, fix $\psi\in C^\infty(\mbR^2;\mbR)$ with $\supp(\psi)\in B(0,c)$ and such that there exists a smooth scalar function $f\in C^\infty(\mbR;\mbR)$ for which
\begin{equation}\label{eq:psi_props}
  \psi(x)= f(|x|).
\end{equation}
Then, fixing a family of sequences $\{\theta^N\}_{N\in \mbN} \subset \ell^\infty(\mbZ^2;\mbR)$, a radius $r>1$, integer $N\in \mbN$ and $k\in \mbZ^2$, we set
\begin{equation*}
    \sigma^N_k(x)\coloneqq  \theta^N_k \frac{1}{r}(\grad^\perp\psi)\left(\frac{x-x^N_k}{r}\right) = \theta^N_k \frac{1}{r} f'\left(\frac{|x-x^N_k|}{r}\right) \frac{(x-x^N_k)^\perp}{|x-x^N_k|}.
\end{equation*}
%
Note that since $c>0$ is finite, there exists some $n(c)\in \mbN$ such that 
\begin{equation}
    \supp\left(\sum_{|k|\geq n(c)}\sigma^N_k\right) \cap \mbT^2 = \emptyset. 
\end{equation}
The time fluctuations are modelled as time varying stochastic weights multiplying each vortex patch, which we choose to be Brownian motions. To keep the noise periodic on $\mbT^2$ these weights must agree for nodes $x^N_k,\,x^N_{k'} \in \Lambda^N$ such that $x^N_k-x^N_{k'}\in \Z^2$. To do so we introduce a sequence of $\mbR$ valued Brownian motions $\{W_t^k\}_k$ such that 
\begin{equation*}
    \big[W^k, W^{k'}\big]_t= \begin{cases}
    t, \qquad \text{if } x_k-x_{k'}\in \Z^2\\
    0 \qquad  \text{otherwise}.
\end{cases}
\end{equation*}
Denote by $\mbF=\{\mcF_t\}_{t\geq 0}$ the filtration generated by the collections of these Brownian motions.
For the same reason, we also fix the real sequence $\{\theta^N_k\}_{k\in \mbZ^2}$ to be such that 
\begin{equation*}
    \theta_k^N = \theta_{k'}^N \qquad \text{ if }  x_k-x_{k'}\in \Z^2.
\end{equation*}
Now it is easy to check that our random velocity field is periodic. Letting $L=(l_1, l_2) \in \Z^2$ be a given shift and defining $k' := k -NL $, we have
\begin{align*}
    \sum_{k\in \Z^2} \sigma_k^N(x + L)W_t^k &= \sum_{k\in \Z^2}\theta_k^N\frac{1}{r}(\grad^\perp\psi)\left(\frac{x + L-k/N}{r}\right)W_t^k  \\
    &= \sum_{k'\in \Z^2}\theta_{(k'+NL)}^N\frac{1}{r}(\grad^\perp\psi)\left(\frac{x - k'/N}{r}\right) W_t^{k'+ NL} \\
    &= \sum_{k'\in \Z^2}\theta_{k'}^N\frac{1}{r}(\grad^\perp\psi)\left(\frac{x - k'/N}{r}\right) W_t^{k'}\\
    &= \sum_{k\in \Z^2} \sigma_k^N(x)W_t^k.
\end{align*}
Next we introduce another important object, the symmetric, correlation matrix 
\begin{equation}\label{eq:A_N_x_y}
    A^N(x,y) \coloneqq \sum_{k\in \mbZ^2} \sigma_k^N(x)\otimes \sigma_k^N(y).
\end{equation}

This matrix field coincide with the spatial covariance of the noise provided that $r< (2c)^{-1}$. Indeed, if $k' = k-NL$ for some $L\in \Z^2_0$ as above, then for all $x\in \mbT^2$ either $\sigma_k(x)$ or $\sigma_{k'}(x)$ is zero, since it cannot be the case that both vectors $x-x_k$ and $x-x_{k'}$ lie in $B(0, rc)$. Thus, in this case, the correlation between $\sigma_k(x)W_t^k$ and $\sigma_{k'}(x)W_t^{k'}$ is zero, even if $[W_t^{k'} , W_t^k]=t$.

 From now on we always assume $r<(2c)^{-1}$. Since, in the end, we will send $r\rightarrow 0$ while keeping $c$ fixed, this condition is no restriction at all.

By a slight abuse of notation we will often omit the superscript $N$ when $N=1$ and write
\begin{equation}\label{eq:A_N_x}
    A^N(x) \coloneqq \sum_{k\in \mbZ^2} \sigma_k^N(x)\otimes \sigma_k^N(x), \qquad A(x)\coloneqq A^1(x).
\end{equation}

We end this subsection by introducing the main scaling assumptions we will use throughout the work. Let us introduce some compact notation. Given any function $f$ and a scaling parameter $r>0$ we define 
\begin{equation}\label{eq:r_scaled_definition}
    (f)_r(\,\cdot\,):= \frac{1}{r}f\left(\frac{\cdot}{r}\right).
\end{equation}
for which we readily have the identity
\begin{equation}\label{eq:r_scaled_vortex_bound}
    \left\|  (f)_r \right\|_{L^2_x(r\T^2)} \, =  \|f\|_{L^2_x(\T^2)}.
\end{equation}
As a result, it is straightforward to observe the following.
\begin{lemma}\label{lem:A_N_rescaling}
    If \begin{equation}\label{ass 0}
        \theta_k^N = \frac{1}{N}   \text{ for all }k\in \Z^2, \tag{A0}
         \end{equation}
         then 
\begin{equation}\label{eq:A_N_scaling_gen_r}
        A^N(x)= (A(Nx))_{Nr} = \sum_{k\in \Z^2}(\nabla^\perp \psi)_{Nr}({Nx-k})\otimes (\nabla^\perp \psi)_{Nr}({Nx-k}).
    \end{equation}
    In particular, if \begin{equation}\label{ass 1}
        \theta^N_{k}= r = \frac{1}{N} \text{ for all }k\in \Z^2 \tag{A1}, 
         \end{equation}
         then
   \begin{equation}\label{eq:A_N_scaling_homog_r}
       A^N(x)= A^1(Nx)
   \end{equation}

\end{lemma}
\begin{proof}
By direct computation, in the first case when $\theta^N_{\,\cdot\,} = \frac{1}{N}$, we have 

    \begin{align*}
        A^N(x)&= \frac{1}{2}\sum_{k\in \Z^2}\frac{1}{N^2 r^2}\nabla^\perp \psi\left(\frac{x-k/N}{r}\right)\otimes \nabla^\perp \psi\left(\frac{x-k/N}{r}\right) \\
        &= \frac{1}{2}\sum_{k\in \Z^2}\left(\frac{1}{N r }\nabla^\perp \psi\left(\frac{Nx-k}{Nr}\right)\right)\otimes \left(\frac{1}{N r }\nabla^\perp \psi\left(\frac{Nx-k}{Nr}\right) \right) \\
       &= \sum_{k\in \Z^2}(\nabla^\perp \psi)_{Nr}({Nx-k})\otimes (\nabla^\perp \psi)_{Nr}({Nx-k}),
    \end{align*}
    which concludes the proof of \eqref{eq:A_N_scaling_gen_r}. It is then clear to see that \eqref{eq:A_N_scaling_homog_r} holds if $\theta^N_{\,\cdot\,}\equiv r= \frac{1}{N}$.
\end{proof}
\subsection{Solution Concepts and Basic Properties}
In order to analyse \eqref{eq:main_strat_N} and its scaling limit, we interpret it in its It\^o form, replacing the Stratonovich integral with an It\^o integral plus the It\^o-Stratonovich corrector. On a formal level, the corrector is given by the covariation
\begin{equation*}
    \frac{\sqrt{2}}{2}\sum_{k\in \Z^2}d[\sigma_k\cdot \nabla u^N, W^k]_t.
\end{equation*}
To formally compute this quantity, we appeal to the divergence free condition on $\sigma_k$, to get
\begin{equation*}
    d(\sigma_k\cdot \nabla u)= \operatorname{div}(\sigma_k du)= dV_t + \sqrt{2}\operatorname{div}((\sigma_k\otimes \sigma_l) \nabla u)dW_t^{l}
\end{equation*}
where $V_t$ is a process with bounded variation. Inserting this expression in the covariation above, and assuming that $r\le (2c)^{-1}$, we see that the It\^o formulation reads
\begin{equation}\label{eq:main_ito_N}
    d u^N_t = \sqrt{2}\sum_{k\in \Z^2} \sigma_{k}^N\cdot \grad u^N_t dW_t^{k} +\Big(\nabla\cdot (A^N\nabla u^N_t) + \kappa\Delta u_t^N \Big)dt 
\end{equation}
Owning to this formal computation, we are motivated to work with the following notion of a solution, working directly with the It\^o form.
\begin{definition}\label{def:ito_N_weak_sol}
    Given $u_0 \in L_0^2(\mbT^2)$ and $T>0$, we say that a process $u^N \in $  $C_\mbF(0,T;L^2_0(\mbT^2))\cap L_\mbF^2(0,T; \mcH^1_0(\mbT^2))$ is a weak solution of \eqref{eq:main_strat_N} on $[0,T]$ if, for every $\phi \in C^{\infty}(\mbT^2)$ it holds that
    \begin{equation}
        \brak{u_t^N, \phi} = \brak{u_0, \phi} + \kappa\int_0^t \brak{u^N_s, \nabla\cdot (\kappa I + A^N)\nabla\phi} ds  + \sum_{k \in \mbZ^2}\int_0^t \sqrt{2}\brak{u^N_s, \sigma_k^N \cdot \nabla \phi}dW^k_s, \quad \mbP\text{-a.s.}
    \end{equation}
\end{definition}
\begin{remark}[Mean Free Test Functions]\label{rem:mean_free_test}
    Note that since \eqref{eq:main_strat_N} and \eqref{eq:main_ito_N} are in divergence form, it directly follows that any weak solution, in the sense of \autoref{def:ito_N_weak_sol} satisfies the identity
    \begin{equation}
        \langle u^N_t,1\rangle_{L^2_x} = \langle u_0,1\rangle_{L^2_x}.
    \end{equation}
    Hence, it is not restriction to additionally require the test function $\phi$ in \autoref{def:ito_N_weak_sol} to be mean free. We will do this consistently from now on.
\end{remark}
An application of the It\^o formula gives the following (see for instance \cite[Thm.~1.2]{flandoli_galeati_luo_22_eddy} \cite[Thm.~1.2]{flandoli_luongo_22_channel} and \cite[Thm.~5.24]{flandoli_spdes}):
\begin{proposition}[Energy Estimate]\label{prop:ito_N_energy}
    Let $u_0 \in L^2_0(\mbT^2)$ and $T>0$. Then, there exists a unique solution to \eqref{eq:main_strat_N} in the sense of \autoref{def:ito_N_weak_sol}.
    Moreover it holds that
    \begin{equation}\label{eq:ito_N_energy}
        \|u^N_T\|^2_{ L^2_x} + 2\kappa \int_0^T \|\nabla u^N_s\|^2_{L^2_x} \dd s = \|u_0\|^2_{L^2_x}, \quad \mbP\text{-a.s.}
    \end{equation}
\end{proposition}
\subsection{Discussion of the Scaling Limit}\label{subsec: discussion}
The first goal of this work is to find suitable scaling regime, in terms of the coefficients $\{\theta^N\}_{N\geq 1}$ and the scaling parameter $r\coloneqq r(N)>0$, such that any sequence of solutions $\{u^N\}_{N\in \mbN}$ to \eqref{eq:main_strat_N} has a meaningful limit as $N\to \infty$. As $N\to \infty$ the distance between the patch centres converges to zero and so we must send both their effective radii $r$ and their intensities, $\{\theta^N_k\}_{k\in \mbZ^2}$ to zero. By scaling these quantities in an appropriate way (\emph{c.f.} Assumption \ref{ass 1}), we will obtain an effective description of the stochastic system in terms of a deterministic parabolic one, in the large $N$ limit. This is done in \autoref{sec: scaling}

The second goal of this work is to investigate the effective diffusivity of this limiting parabolic equation, in particular its behaviour as the molecular diffusivity $\kappa$ is sent to $0$.
In order to shorten the notation, introduce the elliptic operator 
\begin{equation*}
    \mcL^N_\kappa f = \operatorname{div}\left( (\kappa I + A^N(x) \nabla) f\right).
\end{equation*}
As has been well understood in previous works, \cite{galeati_20_convergence,flandoli_galeati_luo_21_delayed,flandoli_luo_21_highmode,flandoli_galeati_luo_22_eddy}, there are two main quantities, derived from parameters of the noise which govern any suitable scaling limit.
\begin{itemize}
    \item The quantity $\eps_N$ that governs the stochastic fluctuations of the solution, which has to be made small (\textit{c.f.} \autoref{lem:A_N_op_est} )
    \begin{equation}\label{eq:eps_N_def}
        \eps_N:= \sup\left\{ \int\int_{\T^4} (\nabla v(x))^t A^N(x,y)\nabla v(y) \dd x \dd y: \  v\in \H^1(\T^2),\, \|\nabla v\|_{L^2}=1\right\}.
    \end{equation}
    \item The quantity $\pi_N^\kappa$, which is the first eigenvalue of the operator $-\mcL ^N_\kappa$ and governs the dissipation of the deterministic part of the It\^o equation
    \begin{equation}\label{eq:pi_N_def}
        \pi_N^\kappa\coloneqq  \inf\left\{ \frac{\int _{\T^2}(\nabla v)^\top(x) (\kappa I +A^N(x,x)) \nabla v(x)\dd x}{\int_{\T^2}|v(x)|^2  \dd x}: \ v\in \dot\H^1, v \neq 0\right\}.
    \end{equation}
    To have a non trivial scaling limit, as in \cite{galeati_20_convergence,flandoli_galeati_luo_22_eddy}, we need to keep $\pi_N^\kappa$ of fixed magnitude while $\eps_N \searrow 0$ (\emph{c.f.} \cref{sec:pde_homogenisation} ).
\end{itemize}
\begin{remark}\label{rem:unif_ell}

It is clear that if we had access to a uniform ellipticity inequality, of the form 
\begin{equation}\label{eq:uniform_A_N_ellipticity}
    \inf_{x\in \T^2} \inf _{\xi\neq 0} \frac{\xi^T A^N(x)\xi^T}{|\xi|^2} \ge M
\end{equation}
then we would immediately have $\pi_N^\kappa \ge C(\kappa+M)$, where $C$ is the Poincaré constant of $\T^2$. 
To the best of the authors' knowledge, all previous works concerning models of this type have been set-up such that  \eqref{eq:uniform_A_N_ellipticity} holds, \cite{galeati_20_convergence,flandoli_galeati_luo_21_delayed,flandoli_galeati_luo_21_mixing,flandoli_galeati_luo_22_eddy}. One useful consequence of our approach via homogenization theory is that it allows us to also study cases when \eqref{eq:uniform_A_N_ellipticity} does not apply, see \autoref{lem:A_cases} and \autoref{rem:A_kernel}.
\end{remark}
\begin{remark}
    Finally, we have to keep in mind that for fixed $N$, even if the new operator was uniformly elliptic, or more in general if it holds $\pi_N^\kappa > m \ge 0$, it would not be evident a priori that the noise in  \eqref{eq:main_strat_N} increases the dissipation rate of the $L^2$ norm of the solution. Indeed, applying It\^o's formula to $\|u^N_t\|_{L^2_x}^2$, the conservative nature of the Stratonovich transport noise gives the same energy balance as for deterministic heat equation $\partial_t u = \kappa \Delta u$. We do not deal with the problem of the convergence of the kinetic energy profile of the solution, which was only very recently tackled in \cite{agresti2024AnomalousDissipationInduced} with tools from stochastic maximal $L^p$-regularity theory; we content ourself with a weaker mixing estimate, identifying a limit \emph{effective diffusivity} coefficient $\pi^\kappa = \lim_{N\rightarrow\infty} \pi_N^\kappa$ and giving a comprehensive analysis of the behaviour of this coefficient across different regimes, hoping in future to close the gap, proving also convergence of the energy profile. 
\end{remark}
\subsection{Main Results}
In this section we state for the reader's convenience the main results of the manuscript. The first one concerns the scaling limit of the SPDE \eqref{eq:main_ito_N} to an homogenized limit, it collects \autoref{prop:weak_conv}, \autoref{cor: diss_hom} and \autoref{prop: diagonality}. 

\begin{theorem}\label{th:main_introduction}
    Let $u_0 \in L^2_0(\mbT^2)$ and $T>0$. For every $c, \kappa >0$, there exists a constant $C(c, \kappa)\ge \kappa$ such that if $\bar{u}$ is the weak solution to 
    $$\partial_t \bar u(t,x) = C(c,\kappa)\Delta \bar u(t,x), \qquad \bar u(0,x)= u_0.$$
     Then, under the scaling assumptions $\theta^N_{\,\cdot\,}\equiv r =\nicefrac{1}{N}$, if $u_t^N$ is the unique weak solution to \eqref{eq:main_ito_N}  on $[0,T]$ with initial condition $u_0$, then, for every $\phi \in C^\infty(\T^2)$, it holds that
    \begin{equation}\label{convergence N_main}
        \lim_{N\rightarrow +\infty} \sup_{t\in [0, T]}  \EE\left[\left|\brak{u^N_t- \bar{u}_t, \phi}\right|^2\right] = 0. 
    \end{equation}
    Moreover, for any $\eps>0$ there exists an $N\coloneqq N(\eps)\in \mbN$ large enough such that for all  $\phi \in C^\infty(\mbT^2)$ and $t>0$
   \begin{equation*}
       \EE \left[\left|  \langle u^N(t,\,\cdot\,),\phi\rangle \right|^2\right] \le 2\left(\eps + \exp (-2C(c, \kappa)t)\right) \expt{\|u^N_0\|^2_{L^2}}.
   \end{equation*}
\end{theorem}
After having established this convergence, we study the effective diffusivity constant $C(c, \kappa)$ in various regimes of the radius $c>0$ (including those for which $A^N(x)$ is degenerate) and $\kappa \ll 1$. 
\begin{theorem}\label{thm: effective_bounds_intro} Under the same assumptions of the previous theorem the following statements hold:
\begin{enumerate}[label=\roman*)]
    \item \label{it:bounds_1_intro} If $c\in (0,\nicefrac{1}{2})$ (\textit{i.e.} the vorticity patches are completely separated) there exists a constant $L>0$, depending only on $c$ such that 
    \begin{equation}
      \kappa \leq  C(c, \kappa) \leq L\kappa \quad \text{ for all }\kappa >0.
    \end{equation}
    In particular, the constant $L$ does not depend on the specific choice of the radial function $\psi$. 
    \item \label{it:bounds_2_intro} If $c= \nicefrac{1}{2}$ then for  any $n\geq \nicefrac{7}{2}$ there exists a $p\coloneqq p(n)\geq 1$ and a constant $L'\coloneqq L'(c,\|\psi\|_{W^{p+1,\infty}_x})>0$ such that
\begin{equation}
      \kappa \leq  C(c, \kappa) \leq L'(\kappa + \kappa^{1-\frac{1}{n}}) \quad \text{for all } \kappa >0.
    \end{equation}
    \item \label{it:bounds_3_intro} If $c \ge \sqrt{2}/2$, then there exists a constant $m>0$ such that
\begin{equation}
    C(c, \kappa) \geq \kappa + m \quad \text{ for all $\kappa >0$.}
\end{equation}
    Moreover, in this case, the constant $m$ can be made arbitrarily large by a suitable choice of the stream function $\psi$ (see \autoref{rem:bounds_scaling} below).
\end{enumerate}
\end{theorem}
We remark that the proof of \autoref{thm: effective_bounds_intro} \ref{it:bounds_3_intro} is actually split into two sub-regimes. Firstly, when $c> \sqrt{5}/2$, the proof is a refinement of the argument in \cite{flandoli_galeati_luo_22_eddy}, which uses that $A^N(x)$ is uniformly elliptic in this regime. Secondly, in the novel and more difficult case  $\nicefrac{\sqrt{2}}{2}\le c \le \sqrt{5}/2$, where $A^N(x)$ has some degeneracies (\emph{c.f.} \autoref{lem:A_cases}), we employ variational arguments coming from the theory of elliptic homogenisation to identify the given lower bound. \\

Finally, in the last part of the work, we present numerical experiments conducted in order to quantify the asymptotic behaviour of $C(c, \kappa)$ for $\kappa \rightarrow 0$, particularly in the regime $c \in (\nicefrac{1}{2},\nicefrac{\sqrt{2}}{2})$ where \autoref{thm: effective_bounds_intro} does not give a conclusive answer.
The main achievement of the numerical analysis predicts that for small values of $\kappa$, the additional diffusivity behaves like $$C(c,\kappa) -\kappa \approx m + \kappa^\alpha,$$ for $m$ and $\alpha$ depending on $c$ in the following way:
\begin{itemize}
    \item In the regime $0<c < \nicefrac{1}{2}$, $m\approx 0$ and $\alpha \approx 1$.
    \item When $c= \nicefrac{1}{2}$ $m \approx 0$ and $\alpha <1$.
    \item In the regime $1/2 < c < \sqrt{2}/2$, $m\approx 0$, while $\alpha\in (0,1)$ (it is not compatible with the endpoints $0$ or $1$).
    \item In the regime $\sqrt{2}/2 < c < \sqrt{5}/2$, $\alpha\in (0,1)$ (it is not compatible with the endpoints $0$ or $1$) while the value of the intercept $m$ is strictly positive.
    \item In the regime $c > \sqrt{5}/2$, 
    $\alpha \approx 1$ and $m$ is strictly positive and grows with the parameter $c$.
    \end{itemize}
We refer to \autoref{sec:numerical_results} for a more in depth discussion. 
%
%
\begin{remark}\label{future work}
As a final comment, let us point out that we believe that the convergences result in \autoref{th:main_introduction} can be improved significantly. For reasons of conciseness and since the main goal of this manuscript is to investigate the properties of the limit homogenized equation, we have proved it in this weak form, which is the least we need to draw a quantitative connection between the SPDE and the homogenized limit. Question like almost sure convergence in the space $C(0, T; L^2)$ (in light of the very recent result \cite{agresti2024AnomalousDissipationInduced}), which would allow us to rigorously compare the dissipation of $L^2(\mbT^2)$ norm of the SPDE with that of the homogenized limit allowing to upgrade our theorem to a result of anomalous dissipation, as well as quantitative homogenization estimates, and convergence rates should be added to the picture, and we aim to do so in a future work. 
\end{remark}
\section{The Scaling Limit}\label{sec: scaling}
In this section we prove our first main result \autoref{th:main_introduction}. The strategy is the following: first we prove a martingale estimate ensuring that for $N$ sufficiently large the noise in \eqref{eq:main_ito_N} is small, and the solution is close to a deterministic counterpart with finite $N$. In the second part we prove homogenization of the deterministic part of the equation to an effective PDE, and finally in the third part we combine the two results to get the main theorem. Notice that in the first step, the martingale estimate, we do not assume any scaling relation between $N$ and $r$ except $N\rightarrow + \infty$, $r\rightarrow 0$ and $\theta_k^N=N^{-1}$. 
\subsection{Controlling the Martingale}
In this section we will assume the scaling relation \eqref{ass 0}.
We first compare the It\^o equation \eqref{eq:main_ito_N} with finite $N\in \mbN$ to a deterministic equation with a specified second order operator
\begin{equation}\label{eq:deterministic_discrete}
	\partial_t \tilde{u}^N_t =  \nabla\cdot \Big(\big(\kappa I + A^N(x)\big) \nabla\tilde{u}^N_t\Big) 
\end{equation}
As above, we define the notion of weak solution to \eqref{eq:deterministic_discrete}.
\begin{definition}\label{def:det_discrete_weak}
	   Given $\tilde{u}_0 \in L^2_0(\mbT^2)$ and $T>0$, we say that a path $ u\in C([0, T]; L^2_0(\mbT^2))$ is a weak solution of \eqref{eq:deterministic_discrete} on $[0,T]$ if, for every $\phi \in C^{\infty}(\mbT^2)$, it holds that
	\begin{equation}
		\langle\tilde{u}_t^N, \phi\rangle = \langle\tilde{u}_0, \phi\rangle + \kappa\int_0^t \langle\tilde{u}^N_s, \nabla\cdot (\kappa I + A^N)\nabla\phi\rangle ds .
	\end{equation}
\end{definition}
Just as argued in \autoref{rem:mean_free_test} it is no restriction to additionally require the test function $\phi$ in \autoref{def:det_discrete_weak} to be mean free.
\begin{proposition}\label{prop:A_N_semigroup}
    For $N\in \mbN$, $\kappa>0$ the following all hold
    \begin{enumerate}[label=\roman*)]
        \item \label{it:A_N_semigroup_1} The operator $u\mapsto -\mcL_N^\kappa u \coloneqq \nabla \cdot ((\kappa I + A^N)\nabla u) $, for $u\in \mcD(\mcL_N^\kappa)=\dot \H^2(\T^2)$ is strongly elliptic.
        \item \label{it:A_N_semigroup_2}The operator $u\mapsto \mcL_N^\kappa u $ is the infinitesimal generator of a analytic semigroup of operators on $H$ for which we write $[0,+\infty)\ni t\mapsto e^{t \mcL_N^\kappa}$. In particular, 
        \begin{equation}
            \|e^{t\mcL_N^\kappa} \phi\|_{L^2_x}\lesssim_T \, e^{-\pi^\kappa_N t} \|\phi\|_{L^2_x}\quad \text{for all }\, \phi \in L^2_0(\mbT^2),
        \end{equation}
        where $\pi_N^\kappa$ is defined by \eqref{eq:pi_N_def}.
        \item \label{it:A_N_semigroup_3} Given $T>0$, a weak solution $u^N$ to \eqref{eq:main_ito_N} and a weak solution $\tilde{u}^N$ to \eqref{eq:deterministic_discrete} on $[0,T]$, then the following identities hold in the sense of distributions
        \begin{equation}\label{eq:main_ito_mild}
            u^N_t = e^{t\mcL_N^\kappa} u_0 + \sqrt{2}\sum_{k\in \mbZ^2}\int_0^t e^{(T-s)\mcL_N^\kappa} \nabla u^N_s \cdot \sigma^N_k \dd W^k_s,\quad \text{for all } \, t\in [0,T]
        \end{equation}
        and
         \begin{equation}\label{eq:deterministic_N_mild}
            \tilde{u}^N_t = e^{t\mcL_N^\kappa} u_0,\quad \text{for all } \, t\in [0,T].
        \end{equation}
    \end{enumerate}
\end{proposition}
\begin{proof}
   See collected results of \cite[Sec.~7.2]{pazy_83_semigroups}. The equality between the weak and the mild form, that is, identity \eqref{eq:main_ito_mild} and the uniqueness of the mild formulation are standard results, see for instance  \cite[Section 3.2]{Flandoli_Luongo_Stochastic}
\end{proof}
%
%
Before continuing, for any $N\in \mbN$, let us introduce the following bi-linear form 
\begin{equation}
   L^2(\mbT^2;\mbR^2)^{\otimes 2} \ni (v,w) \mapsto \mcA^N(v,w) \coloneqq \iint_{\mbT^4} v(x)^\top A^N(x,y) w(y)\, \dd x \dd y \in \mbR.
\end{equation}
so that recalling \eqref{eq:eps_N_def} we have
$$\eps_N = \sup_{v\in \H^1, \|\nabla v\|=1} \mathcal{A}^N(v,v)$$
Given the definition of $A^N$ in \eqref{eq:A_N_x_y} and assumption \eqref{ass 0} we have
\begin{align}\label{eq:cA_N_v_v}
    \mcA^N(v,v) =  \, \iint_{\mbT^4} v(x)^\top A^N(x,y)v(y)\dd x \dd y  = \frac{1}{N^2r^2} \sum_{k\in\,\Box^2_1}\left[\int_{\T^2} v(x)\cdot \nabla^\perp \psi\left(\frac{x-k/N}{r}\right)\dd x\right]^2 
\end{align}
Note that since $A^N$ is a positive semi-definite matrix one automatically has $\mcA^N(v,v) \geq 0$.

The following lemma, relying fundamentally on the assumption of compact support of $\psi$, shows that the operator norm of $v\mapsto \mcA^N(v,v)$ is proportional to the scaling constant $r$ and the $L^2(\mbT^2)$ norm of the basic vortex patch $\nabla^\perp \psi$.
\begin{lemma}\label{lem:A_N_op_est}
     Assume \eqref{ass 0}. For all $v\in L^2(\mbT^2;\mbR^2)$, it holds that
     \begin{equation}\label{eq:A_N_op_est}
      \eps_N=\sup_{N\in \mbN}\mcA^N(v,v) \lesssim_c r^{2} \|v\|^2_{L^2_x} \|\nabla^\perp\psi\|^2_{L^2_x}.
 \end{equation}

\end{lemma}
\begin{proof}
Recalling that $supp(\psi)\subset B(0,c)$, we begin by applying H\"older's inequality to obtain
\begin{equation*}
    \mathcal{A}^N(v,v)\le N^{-2}\sum_{k\in \Z^2}\|v\|^2_{L^2(B(k/N, cr)\cap \T^2)}\|r^{-1}\nabla^\perp \psi(\cdot/r)\|_{L^2_x}^2
\end{equation*}
Recalling the notation \eqref{eq:r_scaled_definition}, it is easy to see that $\|r^{-1}\nabla^\perp \psi(\cdot/r)\|_{L^2}^2 = \|\nabla^\perp\psi\|_{L^2(B(0,c))}^2 $.
Using again the assumption on the supports of $\psi$ we see that for $k\neq k' \in \Z^2$, the supports of $\psi(\frac{x-k/N}{r})$ and $\psi(\frac{x-k'/N}{r})$ have non empty intersection if and only if 
    $$|k-k'|\le 2cNr.$$
Thus, as $N \to +\infty$ and $r \to 0$, for any $x\in \T^2$, $x\neq k/N, \ k \in \Z^2$ the number of overlaps of the supports of the $\sigma_k$ at $x$ 
$$\#\{k \in \Z^2 : |x-k/N|\le cr\}$$
is bounded, up to a constant, by $N^2r^2$. 
This means that there exists a constant $C(c)>0$ for which, using the additivity of the $L^2_x$ norm with respect to union of sets, we can estimate the sum 
\begin{equation*}
    \sum_{k\in \Z^2}\|v(x)\|^2_{L^2(B(k/N, cr)\cap \T^2)} \le CN^{2}r^2 \|v(x)\|_{L^2}^2.
\end{equation*}
Putting all together we arrive at the claimed bound.
%
\end{proof}

The uniform bound of \autoref{lem:A_N_op_est} allows us to obtain the following result. Recall that due to \autoref{rem:mean_free_test} with no loss of generality we can take all test functions to be mean free.
\begin{proposition}\label{prop:weak_conv}
	 Assume \eqref{ass 0}. Let $u_0^N \in L^2\left(\Omega, \F_0, H\right)$, $T>0$ and $u^N,\, \tilde{u}^N$ be weak solutions to \eqref{eq:main_ito_N} and \eqref{eq:deterministic_discrete} with $u_0$ as initial data. Then, there exists a constant $C\coloneqq C(c)>0$ such that for every $\phi \in C^\infty_0(\T^2)$ 
	\begin{equation*}
		\sup_{t\in [0, T]}\mbE\left[ \big\langle u^N_t -\tilde{u}^N_t, \phi\rangle ^2\right] \leq  \frac{C r^{2}}{\kappa} \|\phi\|_{L^\infty}^2\expt[]{\|u_0^N\|_{L^2}^2}
	\end{equation*}
\end{proposition}
\begin{proof}
Defining $D^N_t \coloneqq u^N_t -\tilde{u}^N_t$ we subtract \eqref{eq:deterministic_N_mild} from \eqref{eq:main_ito_mild} to find that for any $\phi\in C^\infty_0(\mbT^2)$ and $t\in [0,T]$ one has
\begin{equation*}
	\langle D^N_t,\phi\rangle  =  \sum_{k\in \Z^2} \int_0^t \left\langle e^{(t-s)\mcL_N^\kappa}\left(\sigma_{k}^N\cdot \nabla u^N_s\right),\phi\right\rangle  dW_s^{k} 
\end{equation*}
Hence, by It\^o's isometry, we have
\begin{equation}\label{ito martingale}
    \mbE\left[ \left|\left\langle D^N_t,\phi\right\rangle\right|^2\right] = \mbE\left[\, \sum_{k\in \Z^2}\int_0^t \left|\left\langle e^{(t-s)\mcL_N^\kappa}\left(\sigma_{k}^N  \cdot \nabla u^N_s\right),\phi\right\rangle\right|^2 \dd s \, \right].
\end{equation}
Define $\phi_{t,s} = e^{(t-s)\mathcal{L}_N^\kappa}\phi$. Using the self-adjointedness of the semigroup and expanding the square inside the integral, we find 
\begin{align*}
    \sum_{k\in \Z^2} \left|\left\langle  e^{(t-s)\mcL_N^\kappa}\left(\sigma_{k}^N \cdot \nabla  u^N_s\right),\phi\right\rangle\right|^2 =&  \sum_{k\in \Z^2} \left|\left\langle \sigma_{k}^N  , \nabla  u^N_s \phi_{t,s}\right\rangle\right|^2\\
=& \sum_{k\in \Z^2} \left\langle \sigma_{k}^N  ,\nabla u^N_s   e^{(t-s)\mcL_N^\kappa}\phi\right\rangle  \left\langle \sigma_{k}^N  ,\nabla u^N_s  e^{(t-s)\mcL_N^\kappa}\phi\right\rangle\\
    = & \iint_{\mbT^4} \left(e^{(t-s)\mcL_N^\kappa}\phi\right)(x)\left(\nabla u^N_s\right)^\top \hspace{-0.5em} (x)  A^N(x,y) \nabla u^N_s(y)   \left(e^{(t-s)\mcL_N^\kappa}\phi\right)(y) \dd x\dd y .
\end{align*}
Now we apply \autoref{lem:A_N_op_est} in the last line to see that 
\begin{align*}
     \sum_{k\in \Z^2} \left|\left\langle  e^{(t-s)\mcL_N^\kappa}\left(\sigma_{k}^N \cdot \nabla  u^N_s\right),\phi\right\rangle\right|^2  &= \mcA^N\left(\nabla u_s^N e^{(t-s)\mcL_N^\kappa}\phi,\nabla u_s^N e^{(t-s)\mcL_N^\kappa}\phi\right) \\&  \le C(c)r^2 \Big\|\nabla u_s^N e^{(t-s)\mcL_N^\kappa}\phi \Big \|_{L^2_x}^2.
\end{align*}
Thus, appealing to  \autoref{prop:ito_N_energy} and \autoref{prop:A_N_semigroup} we obtain the bound 
\begin{align*}
    \mbE\left[ \left|\left\langle D^N_t,\phi\right\rangle\right|^2\right] &\lesssim_c r^{2}\int_0^t \EE\Big[\Big\|\nabla u_s^N e^{(t-s)\mcL_N^\kappa}\phi \Big \|_{L^2_x}^2\Big]ds \\
     &\lesssim_c r^{2}\sup_{s\in [0,T]}\Big\|e^{(t-s)\mcL_N^\kappa}\phi\Big\|_{L^\infty_x}^2\expt[]{\int_0^t\|\nabla u_s^N\|_{L^2_x}^2ds} \\ 
     &\lesssim_c \frac{r^2}{\kappa} \|\phi\|_{L^\infty_x}^2 \expt{\|u_0\|_{L^2_x}^2}.
\end{align*}
 \end{proof}
Once we have this result, given any $u_0 \ge 0$, we readily obtain a quantitative mixing-type estimate for $u^N$.
Recalling the quantity $\pi_N^\kappa$ defined in \eqref{eq:pi_N_def}, which is the first eigenvalue of the elliptic operator $f\mapsto - \mcL^\kappa_N  f =\operatorname{div}\left( (\kappa I + A^N(x) \nabla) f\right)$, we have the following.
\begin{corollary}\label{dissipation piN}
   In the setting of \autoref{prop:weak_conv}, for any $\phi \in C^\infty(\mbT^2)$ there exists a constant $C\coloneqq C(c,\|\phi\|_{L^\infty_x})>0$ such that
\begin{equation}\label{eq:u_N_dissipation}
       \EE \left[\left(  \int_{\T^2} u^N(t,x)\phi(x)\dd x\right)^2\right] \le 2\left( \frac{C r^2}{\kappa} +\exp (-2\pi_N^\kappa t)\right) \expt{\|u^N_0\|^2_{L^2}}.
   \end{equation}
\end{corollary}
\begin{proof}
     First notice that since $u^N$ is mean free, the LHS of \eqref{eq:u_N_dissipation} does not change if we replace $\phi$ by $\phi - m$ for any constant $m\in \R$, so we may assume that $\phi \in C^\infty_0(\mbT^2)$. By \autoref{prop:weak_conv}, we have for every mean free $\phi\in C^\infty_0(\T^2)$, $t\in [0,T]$ and some $C\coloneqq C(c)>0$,
    \begin{equation*}
       \expt[]{\left(\int_{\T^2} \left (u^N(t,x) - \tilde u^N(t,x) \right)\phi(x)\dd x \right)^2} \lesssim  \frac{ C r^2}{\kappa} \|\phi\|^2_{L^\infty_x}\expt[]{\|u_0^N\|_{L^2}^2}. 
    \end{equation*}
    Then, using the inequality $A^2 \le 2(A-B)^2 + 2B^2$, H\"older inequality, along with \autoref{prop:A_N_semigroup} \autoref{it:A_N_semigroup_2} and \autoref{it:A_N_semigroup_3}  we get
     \begin{align*}
        \EE \left[ \left(\int_{\T^2} u^N(t,x) \phi(x)\dd x\right )^2 \right] &\le \|\phi\|_{L^\infty}^2 \left( 2C(c)\frac{r^2}{\kappa} \expt[]{\|u_0^N\|_{L^2}^2} + 2\expt[]{\|\tilde u^N(t,\cdot )\|_{L^2}^2 }\right)\\
        &\lesssim_\phi  2C(c)\frac{r^2}{\kappa} \expt[]{\|u_0^N\|_{L^2}^2} + 2\expt{\|e^{t \mcL ^N} u_0^N\|_{L^2}^2} \\
        & \lesssim_\phi  2C(c)\frac{r^2}{\kappa} \expt[]{\|u_0^N\|_{L^2}^2} + 2 \exp(- 2\pi_N^\kappa t)\expt{\| u_0^N\|_{L^2}^2}
        \end{align*}
\end{proof}
  The quantity $\pi_N^\kappa$, appearing in the estimate is not easy to quantify precisely, especially taking into account its dependence on the molecular diffusivity $\kappa$ and on the radius of support of the vortex patches $c$. As we will see in \autoref{lem:A_cases} and \autoref{prop:pointwise_ellipticity}, if $c$ is chosen sufficiently large, then $A^N(x)$ becomes uniformly elliptic and $\pi_N^\kappa$ can be estimated as in \autoref{rem:unif_ell}, and can be made arbitrarily large by an appropriate choice of the stream-function $\psi$. One could easily prove a weak statement, i.e. that for fixed $\kappa >0$, and every $N>0$, $c>0$, it holds $\pi_N^\kappa > \kappa$. However, in order to obtain sharper bounds or study the limiting behaviours as $N\rightarrow \infty$ and/or $\kappa \to 0$, as is our aim, we turn to homogenization theory and aim to quantify behaviour of the limiting equation satisfied by $\lim_{N\to \infty} u^N$.
\subsection{Homogenization of the Stratonovich Corrector}\label{sec:pde_homogenisation}

In this section, under the scaling regime \eqref{ass 1}, we study the homogenization of solutions $\tilde u^N_t$ to \eqref{eq:deterministic_discrete}, that is we deduce the existence of a constant $C(c,\kappa)\ge \kappa$ such that for $N \rightarrow \infty $ any solution $\tilde u^N_t$ of \eqref{eq:deterministic_discrete} converges in an appropriate sense (\textit{c.f.} \autoref{prop:pde_homogenisation}) to the solution of the homogenized PDE
$$\partial_t \tilde u = C(c,\kappa)\Delta \tilde u.$$
The fact that the homogenized diffusivity matrix is a multiple of the identity is a crucial result of this section (\emph{c.f.} \autoref{prop_diag})
Let us introduce the notation 
\begin{equation*}
    H_\kappa(x):= \kappa I + A(x).
\end{equation*}
So that we rewrite the equation \eqref{eq:deterministic_discrete} for $\tilde u_t^N$ as
\begin{equation*}
    \partial_t \tilde u^N (t,x)= \nabla \cdot \Big(H_\kappa(Nx) \nabla \tilde u^N(t,x)\Big).
\end{equation*}
We work under the assumptions of \autoref{lem:A_N_rescaling}, that is we set $\theta^N_{\,\cdot\,}\equiv r=\nicefrac{1}{N}$. As a result we always have $A^N(x)= A^1(Nx)=A(Nx)$ by \autoref{lem:A_N_rescaling},  which allows us to use standard homogenization results. Recall also that, if $\kappa>0$, we have the uniform ellipticity property
\begin{equation}\label{eq:H_uniform_elliptic}
    H_\kappa(x) \xi \cdot \xi = \left(\kappa I + A(x)\right) \xi \cdot \xi \ge \kappa|\xi|^2 \quad\text{for \,  a.e. }\,  x \in \T^2
\end{equation}

 Moreover, since $H_\kappa$ is time independent, the homogenization of the parabolic problem reduces to homogenization of the elliptic operator $H_\kappa$ (see \cite[Ch.~2, Rem.~1.6]{bensoussan1978AsymptoticAnalysisPeriodic}). \\
 
The following theorem is the main result of this section, which is an application of  \cite[Chapter 2]{bensoussan1978AsymptoticAnalysisPeriodic} to our case.
\begin{proposition}\label{prop:pde_homogenisation}
    Assume \eqref{ass 0}. Let $u_0\in \dot L^2(\mbT^2)$, $\kappa >0$, $T>0$ and $\tilde u^N$ be the associated unique weak solution of \eqref{eq:deterministic_discrete} in the sense of \autoref{def:det_discrete_weak}.
    For $e_1,\,e_2$ as above and $\phi_i$ for $i=1,\,2$ the unique solutions of
    \begin{align}\label{eq:corrector}
    \begin{cases}
        \nabla\cdot \Big(H_\kappa e_i+ H_\kappa \nabla \phi_i\Big) =0 \\
        \phi_i \in \dot \mcH^1(\mbT^2).
    \end{cases}
\end{align}
we let $\bar H_\kappa$ be the matrix defined by
\begin{equation} \label{eq:hom_H}
    (\bar H_{\kappa})_{ij} := \int_{\mbT^2} (H_\kappa(e_j + \nabla \phi_j))\cdot e_i. 
\end{equation}
    Then, finally, letting $\bar u \in C(0, T; \dot \H^1(\T^2))$ be the unique weak solution to the initial value problem 
    \begin{equation}\label{eq:homogenised_pde}
\partial_t \bar u = \nabla \cdot (\bar H \nabla \bar u), \qquad \bar u(0,x)= u_0        
    \end{equation}
one has the limits
\begin{equation*}
    \tilde u^N \to \bar u  \,\, \text{ in } L^2([0, T]; L^2(\T^2)) \quad  \text{and }\quad  \tilde u^N \rightharpoonup \bar u\,\, \text{in } L^2([0, T]; \mcH^1(\mbT^2)).
\end{equation*}
\end{proposition}
\begin{proof}
We refer the reader to \cite[Chapter 2, Thm.~2.1 \& Rem.~1.6]{bensoussan1978AsymptoticAnalysisPeriodic}, the adaptation to the periodic setting is straightforward.
\end{proof}
Let us introduce the Frechet space $\dot \mcH^{-}(\mbT^2) \coloneqq \bigcap_{\alpha>0} \dot\mcH^{-\alpha}(\mbT^2)$, endowed with its natural topology induced by the sequence of seminorms $(\|\|_{\dot \H^\alpha})_{\alpha >0}$
\begin{corollary}\label{H-}
    In the above setting it also holds that $\tilde u^N \rightarrow \bar u$ in $C([0, T]; \mcH^-(\mbT^2))$. As a result, for any $\phi \in C^\infty(\mbT^2;\mbR)$
    \begin{equation}
       \lim_{N\to\infty} \sup_{t\in [0,T]}\left|\langle \tilde{u}^N -\bar{u},\phi\rangle\right| = 0.
    \end{equation}
\end{corollary}
\begin{proof}
   Thanks to the boundedness of $A$, it is not difficult to see that $\tilde u_t^N \in L^\infty([0, T]; L^2(\mbT^2))$ and $\partial_t \tilde u^N \in L^2([0, T]; \mcH^{-1}(\mbT^2))$, thus it follows by the Aubin-Lions lemma that $\tilde u^N_t$ is pre-compact in $C([0, T]; \mcH^{-\eps}(\mbT^2))$ for every $\eps>0$. By the uniqueness of the limit, the whole sequence must also converge in $\mcH^-(\mbT^2)$. 
\end{proof}
As an important consequence of our specific set-up we can show that the homogenized matrix $\bar{H}$ is in fact diagonal. We defer the proof of this fact to \autoref{app_A1}. 
\begin{proposition} \label{prop: diagonality}
   For $H_\kappa = \kappa I + A$ and $\bar{H}_\kappa$ as defined by \eqref{eq:hom_H} it holds that $\bar H_k = C(c, \kappa) I$ for some positive constant $C(c, \kappa)\ge \kappa$.
\end{proposition}
\begin{proof}
    See \autoref{sec:homogenised_diagonallity}.
\end{proof}
\subsection{Combined Homogenised and It\^o-Stratonovich Diffusion Limit}\label{sec:combined_homog_ito_limit}
In the previous sections we have shown that the solution $u^N$ to \eqref{eq:main_strat_N} is close, in a weak sense and for $N$ large, to the solution $\tilde u_t$ of a deterministic problem with an additional elliptic operator given by $\nabla\cdot (A(Nx)\cdot \nabla \tilde u_t)$, \eqref{eq:deterministic_discrete}. We have also proved that for $N$ large, $\tilde u^N$ approaches a homogenized limit $\bar{u}$. Combining the results, we obtain the following.
\begin{theorem}\label{th:spde_homogenisation}
    Assume \eqref{ass 1}. Let $u_0 \in L^2_0(\mbT^2)$, $T>0$ and $\bar{u}$ be the associated weak solution to \eqref{eq:homogenised_pde} as in \autoref{prop:pde_homogenisation}. Under our fixed scaling assumptions ($\theta^N_{\,\cdot\,}\equiv r =\nicefrac{1}{N}$), if $u^N$ is the unique weak solution to \eqref{eq:main_ito_N}  on $[0,T]$ with initial condition $u_0$, then, for every $\phi \in C^\infty(\T^2)$, it holds that
    \begin{equation}\label{convergence N}
        \lim_{N\rightarrow +\infty} \sup_{t\in [0, T]}  \EE\left[\left|\brak{u^N_t- \bar{u}_t, \phi}\right|^2\right] = 0. 
    \end{equation}
\end{theorem}
\begin{proof}
    By \autoref{prop:weak_conv},
    \begin{equation*}
        \sup_{t\in [0, T]} \EE\left[\left|\brak{u^N_t- \tilde u_t^N, \phi}\right|^2\right] \le  \frac{1}{N^2\kappa}\|u_0\|_{L^2}^2.
    \end{equation*}
    Adding and subtracting $\tilde u_t^N$ inside \eqref{convergence N} we have 
    \begin{align}
     \sup_{t\in [0, T]}   \EE\left[\left|\brak{u^N_t- \bar{u}_t, \phi}\right|^2\right] &\le 2\sup_{t\in [0, T]} \EE\left[\left|\brak{u^N_t- \tilde u_t^N, \phi}\right|^2\right] + 2\sup_{t\in [0, T]}\left|\brak{\tilde u^N_t- \bar{u}_t, \phi}\right|^2 \\
        &\lesssim \frac{1}{N^2\kappa}\|u_0\|_{L^2}^2 + 2 \sup_{t\in [0, T]}\left|\brak{\tilde u^N_t- \bar{u}_t, \phi}\right|^2.
    \end{align}
   Thus, thanks to \autoref{H-}, for every $\eps>0$ we can find an $N_0\coloneqq N_0(\eps)\in \mbN$ such that for every $N>N_0$, 
   $$\sup_{t\in [0, T]}\left|\brak{\tilde u^N_t- \bar{u}_t, \phi}\right|^2 \le \eps$$
   we conclude that 
   \begin{equation*}
       \limsup_{N\rightarrow +\infty} \sup_{t\in [0, T]}\EE\left[\left|\brak{u^N_t- u_t, \phi}\right|^2\right] \le \eps
   \end{equation*}
   and the result follows by the arbitrariness of $\eps$. 
\end{proof}
As a result of \autoref{th:spde_homogenisation} we can improve the mixing estimate given by \autoref{dissipation piN}, replacing the role of $\tilde{u}^N$ therein with $\bar{u}$ which is crucially independent of $N$. 
\begin{corollary}\label{cor: diss_hom}
 Assume \eqref{ass 1}, then for every $\eps >0$, there exists $N$ large enough such that 
   \begin{equation*}
       \EE \left[\left(  \int_{\T^2}u^N(t,x)\phi(x)\dd x\right)^2\right] \lesssim_\phi 2\left(\eps + \exp (-2C(c, \kappa)t)\right) \expt{\|u^N_0\|^2_{L^2}}.
   \end{equation*}
\end{corollary}
\begin{proof}
   Follow exactly the proof of \autoref{dissipation piN}, only replacing the pivot around $\tilde{u}^N$ therein with one around $\bar{u}$ and appealing to \eqref{convergence N} along with \autoref{prop: diagonality}.
    %
\end{proof}
\begin{remark}
    The statement of this corollary contrary to \autoref{dissipation piN} is not quantitative in $N$. Obtaining a quantitative estimate is possible, see for instance \cite{shen_2018_periodic}, but would have taken us away from the main scope of this manuscript, therefore we reserve a refined version of this statement for future works. 
\end{remark}
\begin{remark}
   It is a classic result that it actually holds $\pi_N^\kappa\rightarrow C(c, \kappa)$ as $N\rightarrow +\infty$. This fact can be proven by means of $\Gamma$-Convergence, employing the variational setting that we introduce in \autoref{sec: variational}. We refer the reader to \cite[Chapter 24]{maso1993introduction}, and in particular Theorem 24.1. 
\end{remark}
\section{Analysis of Homogenised Diffusivity}\label{sec:diffusivity}
\autoref{th:spde_homogenisation} and \autoref{prop: diagonality} shows that our stochastic equation \eqref{eq:main_ito_N} is arbitrarily close to an heat equation \eqref{eq:homogenised_pde} with diffusivity $C(c, \kappa)$. The goal of the rest of this paper is to give quantitative estimates on $C(c, \kappa)$.  Recall that \autoref{lem:homogenised_upper_lower_symmetric} shows that the homogenized diffusivity $C(c, \kappa)$ satisfies

\begin{equation}\label{eq:homogenised_turb_linear_lower}
            C(c,\kappa)\geq  (\kappa +\lambda_A)
        \end{equation}
where the factor $$\lambda_A = \inf _{x\in \T^2}\sup _{\xi \neq 0}\frac{\xi^T A(x)\xi}{|\xi|^2}$$ is the uniform ellipticity constant of $A(x)$. This additional factor, as shown in the previous sections, is due to the spatial structure of the noise, and can be made arbitrarily large, as recalled in the introduction, by letting the supports of the vector fields $\sigma_k^N$ overlap sufficiently and have large enough intensities (\textit{c.f.} \autoref{prop:pointwise_ellipticity}). However, we additionally consider situations in which $\lambda_A=0$ and demonstrate that, in some regimes, some enhancement of dissipation ($C(c,\kappa) \ge (\kappa + m)$) is still possible. Secondly, we investigate the asymptotics of $C(c, \kappa)$ as $\kappa \rightarrow 0$, identifying different regimes based on the level of overlap of the vortex patches.  
See \cite{fannjiang_papanicolaou_94} for a similar situation with a carefully constructed cellular transport.

Before proceeding with this analysis we study finer properties of the matrix $A^N(x)$ for varying $c>0$. This is the content of the next subsection.
\subsection{Properties of the covariance Matrices}
We discuss some properties of the matrices $A^N(x)$. 
We remark that if $c\le 1$, that is each vortex patch has support contained in the unit ball, then the only non-zero summands in the definition of $A(x)$ are those indexed by $k\in\,\Box^2_1 = \{ (0,0), \ (1,0), \ (0,1), \ (1,1)\}$. On the other hand, if $c\ge 1$, then there are contributions from vector fields with centres outside the unit square. However, in both cases $A(x)$ corresponds to the periodization of one single vortex patch of radius $c>0$.

In order to understand ellipticity properties of $A(x)$, we begin by studying its degeneracies. For every $x\in \T^2$, the matrix $A(x)$ is positive semidefinite by definition, so we are interested in the regions where it fails to have full rank. We also keep in mind that when $c\leq1$, $A(x)$ has zeroes at the lattice points $k\in \Z^2$ as a consequence of smoothness radial symmetry and the support of $\psi$.
\begin{figure}[H]
    \begin{center}
      \setlength{\unitlength}{\textwidth}
        \includegraphics[width=.8\textwidth]{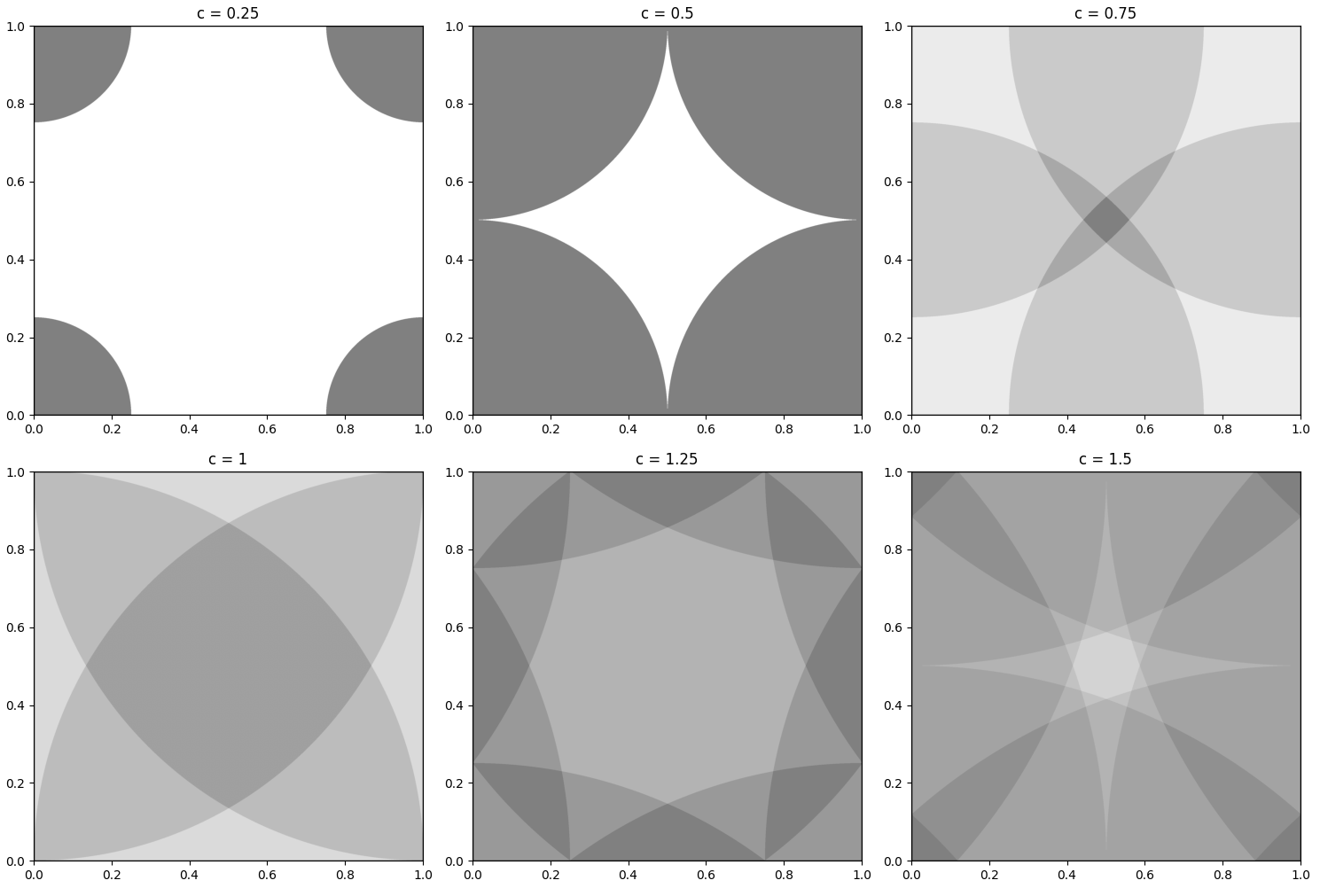}   
            \end{center}
    \caption{Illustration of the supports of $A(x)$ for selected $c \in \{0.25,\, 0.5,\, 0.75,\, 1,\, 1.25,\, 1.5\}$. Darker shading represent multiple overlaps.}
\label{fig:add_visc} \end{figure}
\begin{lemma}\label{lem:A_cases}
    The following situations all hold.
    \begin{enumerate}
        \item If $c = 1/2- \delta$ for some $\delta\in (0,\nicefrac{1}{2})$ then 
        \begin{equation*}
          \Big\{(x_1, x_2)\in \T^2 : \min(|x_1 - 1/2|, |x_2 - 1/2|)\le 2\delta \Big\} \subset  \supp(A)^c.
        \end{equation*}
        \item If $1/2 \leq c < \nicefrac{\sqrt{2}}{2}$, then letting $\delta = \nicefrac{\sqrt{2}}{2}-c$, it holds that
        \begin{equation*}
            \Big\{x\in \T^2 : \min_{k\in \{(0,0),(0,1), (1,0), (1,1)\}}(|x-k|)\ge \frac{\sqrt 2}{2}-\delta \Big\} =   \supp(A)^c.
        \end{equation*}
        \item If $\sqrt{2}/2 \leq c \le 1$ then letting $\delta = 1-c$ it holds that $A$ has a null eigenvalue in the region 
        \begin{equation*}
          \Big \{x\in \mbT^2\,:\,  |x-k| \ge 1-\delta \text{ for exactly three $k$ out of }(0,0), (0,1), (1,0), (1,1) \Big\}
        \end{equation*}
        and on the edges of the square $\partial([0,1]^2)$, while it has full rank in the rest of $[0,1]^2$.
        \item Let $1 \leq c \leq \nicefrac{\sqrt{5}}{2}$, then setting $\delta= \nicefrac{\sqrt{5}}{2} -c$ it holds that $A$ has full rank outside of the region 
        %
%
     \begin{equation}
    \bigg\{x \in \mbT^2 \,:\, \exists \, a \in  B\bigg(\nicefrac{1}{2},\sqrt{\nicefrac{1}{4}+ \delta^2-\sqrt{5}\delta}\bigg)  \,\, \text{s.t. }\,  x \in \{(a, 0),\,(a, 1),\, (0, a),\,(1, a)\} \,\, \bigg\}.
\end{equation}
        Furthermore, in this region it has one null eigenvalue.
        \item If $c > \sqrt{5}/2$ then $A$ is strictly positive definite everywhere on $\T^2$. \label{it:A_regions_5}
    \end{enumerate}
\end{lemma}
\begin{proof}
    Before proving each statement it is useful to picture the configuration of the patches $\sigma_k$. We have one patch on each integer lattice point $k\in \Z^2$. As we already noted, at each location $x \in \T^2$ only a finite number of vortices are active, each of which has support exactly equal to the ball of radius $c$. For instance, if $c < 1$ it is enough to consider only four patches located at the corners of the square. We will derive our statements using geometric considerations of the intersections of the patches supports. The regions where $A\equiv 0$ are those that lie outside of the support of all vortices, so they can only be present when the torus is not fully covered by balls of radius $c$ located at integer lattice points.
    \begin{enumerate}
        \item By our previous considerations, it is straightforward to see that the region in the first statement is made of two strips that separate the supports of the patches $\sigma_k$ located at the corners of the square. 
        \item In the second situation, the supports of the patches do overlap  but the radius $c$ is too small to reach the centre of the square, leaving out the diamond shaped region described in the statement, where $A\equiv 0$. 
        
        \item In this situation the torus is fully covered. The matrix field is again given by 
        \begin{equation}
            A(x)= \sum_{k\in\,\Box^2_1}\nabla^\perp \psi({x-k})\otimes \nabla^\perp \psi({x-k}).
        \end{equation}
        The matrix $A$ is the sum of orthogonal projections along $\nabla^\perp\psi(x-k)$ for $k\in \Z^2$. 
        Notice that the region described in the statement of the lemma is exactly the region where only one patch out of the four is active, thus $A$ is a single projection, and cannot have full rank. In the remaining region, either two three or four patches are active. Recall that $\psi$ is radial, and thus, for every $x\in \T^2$, $\nabla^\perp\psi(x-k)$ is parallel to $(x-k)^\perp$. In the region where only two vortices are active, $x-k_1$ and $x-k_2$ are linearly independent unless $x$ lies in the span of $k_1-k_2$, which happens only on the edges of the square. Where either three or four patches are active, we argue in the following way. Suppose, for a contradiction, that there exists $w\in \R^2\setminus \{0\}$ such that $w\cdot  A(x)w=0$. This would mean that $\sum_{k\in\,\Box^2_1}|\nabla^\perp\psi(x-k)\cdot w|^2 =0$ but this can happen if and only if $\omega \perp (x-k)$ for all active $k$. We conclude that if the $k$ are not all collinear, we must have $w=0$. 
        
        \item In this case we reason as above, but we can prove that the system of vectors $\{\nabla^\perp\psi(x-k)\}_{k\in \Box^2_1}$ spans $\R^2$ for every $x\in \T^2$ except on the edges of the square. Since $c>1$, at each point $x\in \T^2$ there are at least two or three active patches who's centres are not co-linear. Since $A(x)$ must have full rank wherever there are more than two active vortices we are left with the region where only two vortices are active. Reasoning as in the previous case, $A(x)$ has rank $1$ in this region only along the segment connecting the centres of the blobs. This region is a subset of the edges of the square described as in the statement.
        \item Finally, in the last situation, at each point of the square there are active at least three active vortices with centres that are not collinear and hence $A$ has full rank everywhere.
    \end{enumerate}
\end{proof}
\begin{remark}\label{rem:A_kernel}
    The proof of \autoref{lem:A_cases} allows us to exhibit the null eigenspace of the matrix $A(x)$ in the regime $1\le c < \sqrt{5}/2$: it coincides with the direction parallel to the boundary of the square $\partial [0,1]^2$, i.e. for $1\le c < \sqrt{5}/2$
    \begin{align*}
        A(x)e_1 &= 0 \qquad \text{for }\ x=(x_1, 0) \text{ or  } \, x=(x_1, 1), \\
    A(x)e_2 &= 0 \qquad \text{for }\ x=(0, x_2) \ \text{ or }\ x=(1, x_2).
    \end{align*}
    Moreover, for $c\ge 1$, and $x$ in the set of points where $A$ is invertible, there is a simple expression for the inverse of $A$ given by 
    \begin{equation}\label{A inverse}
        A^{-1}(x)= \frac{1}{\det(A(x))}\sum_{k\in \mbZ^2}\nabla \psi(x-k)\otimes \nabla \psi(x-k)
    \end{equation}
    Indeed, for a symmetric matrix $A=(a, b; b , c)$, the inverse is given by $A^{-1}= (\det A)^{-1} (c, -b; -b, a)$
\end{remark}
In the scenario \ref{it:A_regions_5} of \autoref{lem:A_cases} we can in fact quantify the ellipticity of $A$. First we require a simple geometric lemma, whose proof is deferred to \autoref{A1_b}.
\begin{lemma}\label{lem:geometry_fact}
    For any $x\in \T^2$ and $v\in \mathbb{S}^1$ it is always possible to find an integer point $\bar k \in \Box^2_1$ such that
   \begin{equation}\label{eq:ellipticity_geometry_fact}
        \Bigg|\, v \cdot \frac{(x-\bar k)^\perp}{|x-\bar k|} \Bigg|> \frac{1}{\sqrt{2}}.
    \end{equation}
   Moreover, $\bar k$ can always be chosen such that $|x- \bar k|\le \sqrt{5}/2$.
\end{lemma}
\begin{proof}
    See \autoref{A1_b}.
\end{proof}
\begin{proposition}\label{prop:pointwise_ellipticity}
    Let $\theta_k^N\equiv r=\frac{1}{N}$  and $c>\sqrt{5}/2$. Then, there exists a constant $M$ such that for every $\xi \in \R^2\setminus \{0\}$ 
    \begin{equation}
        A^N(x)\,  \xi \cdot \xi \ge M|\xi|^2,\quad \text{for all }x\in \mbT^2.
    \end{equation}
    Moreover, it holds that 
    \begin{equation}\label{eq:little_m_def}
        M\geq m\coloneqq  \inf\big\{|f'(y)|, \, y \in B(0, \sqrt{5}/2)\big\}.
    \end{equation} 
\end{proposition}
\begin{proof}
   Recall that by definition,
    \begin{equation*}
       A(x)\xi \cdot \xi \coloneqq A^1(x) \xi \cdot \xi = \sum_{k\in \Box_1^2} |\nabla^\perp\psi(x-k) \cdot \xi|^2
    \end{equation*}
   and that, with our choice of $\theta^N_k$ and $r$, we may apply \autoref{lem:A_N_rescaling} to see that $A^N(x)= A(Nx)$. Therefore, it is enough to prove the statement for $N=1$. 
   
   Since $\nabla^\perp\psi(x-k) =\frac{(x- k)^\perp}{|x- k|} f' (|x-k|) $ we may appeal to \autoref{lem:geometry_fact} to find
   \begin{equation*}
   A(x)\xi\cdot \xi   = \sum_{k\in \mbZ^2} \left|\frac{(x- k)^\perp}{|x- k|}\cdot \xi \, f' (|x-k|) \right|^2 \ge \frac{1}{2} \inf_{k\in \mbZ^2}|f'(x-k)|^2. \end{equation*}
   
   Due to the assumed support of $f$, the minimal requirement for the last quantity to be bounded away from zero is that $\sup_{k\in\Box^2_2}|x-\bar k| < c$, which is guaranteed by $c>\nicefrac{\sqrt5}{2}$ and the last statement of Lemma \autoref{lem:geometry_fact}. In addition, this gives the stated lower bound of $M$ by $m \coloneqq \inf\big\{|f'(y)|, \, y \in B(0, \nicefrac{\sqrt{5}}{2})\big\}$.
\end{proof}
%
%

%
%
\subsection{Variational Characterisation of the Homogenised Operator}\label{sec: variational}
Since the matrix $H_\kappa = \kappa I + A$ is symmetric (c.f. \eqref{eq:A_N_x_y}), we can employ a variational approach to study the homogenized coefficient $C(c, \kappa)$.  Let $\xi \in \mathbb{S}^1$ be unit vector, so that by \autoref{prop: diagonality} the homogenized quadratic form  $\overline{H}_\kappa = C(c, \kappa )I$ (recall~\eqref{eq:hom_H}) is characterized, by the expression
\begin{equation}
    C(c,\kappa) =  \xi \cdot \int_{\T^2}H_\kappa(x)(\xi + \nabla \phi_\xi(x))\  \dd x 
    \label{eq:total_viscosity}
\end{equation}
where $\phi_\xi$ solves the corrector equation \eqref{eq:corrector} with $\xi$ in place of $e_i$. Using the fact that $\phi_\xi$ solves the cell problem \eqref{eq:corrector}, and integration by parts, we can write 
\begin{equation}
   C(c,\kappa) = \int_{\T^2}|H^{1/2}_\kappa(x)(\xi + \nabla \phi_\xi(x))|^2 \dd x
    \label{eq:total_viscosity_var}
\end{equation}
For any $\xi\in \mbS^1$ we introduce the convex functional,
\begin{equation*}
    \mcH^1(\T^2)\ni u \mapsto \mcE_\kappa(u)= \int_{\T^2}|H^{1/2}_\kappa(x)(\xi + \nabla u)|^2 \dd x.\end{equation*}
we deduce from the Euler-Lagrange equation associated to $\E_\kappa$,  that one has
\begin{equation*}\E_\kappa(\phi_\xi)= \min_{u\in \H^1(\T^2)}\mcE_\kappa (u) \quad\text{and}\quad  \phi_\xi = \argmin_{u\in \H^1(\T^2)} \E_\kappa(u).
\end{equation*}
That is, the unique minimizer of the energy $\E_\kappa$ coincides with the solution of the cell problem in a given direction $\xi$ and the energy associated to this minimizer, as a function of $\xi$, defines the quadratic form associated to $\overline H_\kappa$.
%
%

%
\subsection{Characterisations of Additional Diffusivity}\label{sec:additional_viscosity}
Recall from \autoref{lem:A_cases}  that ellipticity of the matrix $A$  depends on the value of $c$ in particular, for $c \in (0,\nicefrac{\sqrt{2}}{2})$ the support of $A$ is of measure strictly less than that of the torus. At the other extreme, for $c>\nicefrac{\sqrt{5}}{2}$ the matrix $A$ is strictly positive definite everywhere on $\mbT^2$. Furthermore, it was shown in \autoref{prop:pointwise_ellipticity} that in this latter regime, there is a quantifiable lower bound on the ellipticity of $A$. Therefore, due to \autoref{lem:homogenised_upper_lower_symmetric} we  have two direct conclusions; for $c \in (0,\nicefrac{\sqrt{2}}{2})$ the matrix $\bar{H}_\kappa - \kappa I$ is non-negative definite for all $\kappa > 0$, while for $c > \nicefrac{\sqrt{5}}{2}$ the matrix $\bar{H}_\kappa - \kappa I$ is positive definite with a quantifiable lower bound. Beyond these direct conclusions, however, we are interested in obtaining finer properties of the quantity 
\begin{equation}\label{eq:additional_viscosity}
  \nu_\kappa\coloneqq C(c, \kappa) - \kappa =\left(\bar{H}_\kappa \xi \cdot \xi - \kappa \right), \qquad \text{for any } \, \xi \in \mathbb{S}^1 
\end{equation}
in the limit $\kappa \to 0$. Recall that this quantity does not depend on the choice of $\xi \in \mathbb{S}^1$. We refer to the quantity $\nu_\kappa$ as the \emph{additional diffusivity} in the homogenized limit. This section will give a proof of \autoref{thm: effective_bounds_intro}, whose statement we now recall.
\begin{theorem}\label{thm:effective_bounds_body} Under the same assumptions of the previous theorem the following statements hold:
\begin{enumerate}[label=\roman*)]
    \item \label{it:bounds_1_body} If $c\in (0,\nicefrac{1}{2})$ (\textit{i.e.} the vorticity patches are completely separated) there exists a constant $L>0$, depending only on $c$ such that 
    \begin{equation}
      \kappa \leq  C(c, \kappa) \leq L\kappa \quad \text{ for all }\kappa >0.
    \end{equation}
    In particular, the constant $L$ does not depend on the specific choice of the radial function $\psi$. 
    \item \label{it:bounds_2_body} If If $c= \nicefrac{1}{2}$ then for  any $n\geq \nicefrac{7}{2}$ there exists a $p\coloneqq p(n)\geq 1$ and a constant $L'\coloneqq L'(c,\|\psi\|_{W^{p+1,\infty}_x})>0$ such that
\begin{equation}
      \kappa \leq  C(c, \kappa) \leq L'(\kappa + \kappa^{1-\frac{1}{n}}) \quad \text{for all } \kappa >0.
    \end{equation}
    \item \label{it:bounds_3_body} If $c \ge \sqrt{2}/2$, then there exists a constant $m>0$ such that
\begin{equation}
    C(c, \kappa) \geq \kappa + m \quad \text{ for all $\kappa >0$.}
\end{equation}
    Moreover, in this case, the constant $m$ can be made arbitrarily large by a suitable choice of the stream function $\psi$ (see Remark~\ref{rem:bounds_scaling} below). 
\end{enumerate}
\end{theorem}
\begin{remark}\label{rem:bounds_scaling}
    In the setting of Theorem~\ref{thm:effective_bounds_body} \ref{it:bounds_3_body} we observe the following, not necessarily sharp, quantifications of the relationship between $m$ and $\psi$.
\begin{enumerate}[label=\alph*)]
        \item \label{it:bounds_scaling_1} If $\nicefrac{\sqrt{2}}{2} < c$ then $m$ is $1$-homogeneous in $\psi$ (i.e. $m^\lambda$ associated to $\lambda \psi$ satisfies $m^\lambda = \lambda m^1$.)
        \item \label{it:bounds_scaling_2}  If $\nicefrac{\sqrt{5}}{2} <c $ then we have $m\geq  \inf \{|f'(y)|, y \in B(0,\nicefrac{\sqrt{5}}{2})\}$.
    \end{enumerate}
One can check \ref{it:bounds_scaling_1} by the argument discussed in \autoref{rem:c_greater_root_2_over_2_scaling} while \ref{it:bounds_scaling_2} follows from a combination of \autoref{prop:pointwise_ellipticity} and \autoref{lem:homogenised_upper_lower_symmetric}. 
\end{remark}

\subsubsection{The regime $c \in (0, \nicefrac{1}{2}]$}
%
%
In this regime, the matrix $A$ is not uniformly elliptic, i.e. there exist a set of positive measure in $ \mbT^2$ where $A \equiv 0$, therefore the trivial lower bound from \autoref{lem:homogenised_upper_lower_symmetric} gives that for all $\xi\in \mbR^2$ and $\kappa >0$, 
\begin{equation}
    C(c, \kappa)  \geq \kappa  \quad \implies  \quad  \nu_\kappa \geq 0.
\end{equation}
In particular, $\lim_{\kappa \to 0} \nu_\kappa(\xi) \geq 0$. We are interested in constructing a commensurate upper bound on $\nu_\kappa(\xi)$. We will achieve this by constructing approximate solutions to the cell problem for the unit vector $e_1 =(1,0)$.

\begin{lemma}
Let $c\in (0,\nicefrac{1}{2}]$. Then,
\begin{itemize}
    \item for $c\in (0,\nicefrac{1}{2})$ there exists a constant $L>0$ depending on $c>0$ alone such that
    \begin{equation}
       \nu_\kappa  \leq L \kappa.
    \end{equation}
    \item for $c= \nicefrac{1}{2}$ and any $n\geq \nicefrac{7}{2}$ there exists a $p\coloneqq p(n)\geq 1$ and a constant $L'\coloneqq L'(c,\|\psi\|_{W^{p+1,\infty}_x})>0$ such that 
    \begin{equation}
         \nu_\kappa  \leq L' \left(\kappa+\kappa^{1-\frac{1}{n}}\right).
    \end{equation}
\end{itemize}
\end{lemma}
\begin{proof}
By the definition of $H_\kappa$ and $C(c,\kappa)$ (recall \eqref{eq:total_viscosity_var}) and \autoref{prop: diagonality}, for any $\xi \in \mathbb{S}^1$, we have
\begin{equation*}
\begin{aligned}
    \nu_\kappa &= \left(\bar{H}_\kappa \xi \cdot \xi - \kappa \right) \\
    &= \inf_{\phi \in H^1_0(\mbT^2)} \left\{ \int_{\mbT^2} \left[\kappa |\nabla \phi(x)|^2 + (\nabla \phi(x))^\top A(x)\nabla \phi  + 2(\nabla \phi(x))^\top A(x) \xi + \xi A(x) \xi \right] \dd x;\,  \right\},
    \end{aligned}
\end{equation*}
 So that for any $\phi \in H^1_0(\mbT^2)$
 \begin{equation}\label{eq:turb_visc_upper_1}
 	\nu_\kappa \leq \int_{\mbT^2} \left[\kappa |\nabla \phi(x)|^2 + (\nabla \phi(x))^\top A(x)\nabla \phi(x)  + 2(\nabla \phi(x))^\top A(x) \xi + \xi A(x) \xi  \right] \dd x.
 \end{equation}
 Furthermore, observing that 
 \begin{equation}
 	(\nabla \phi(x))^\top A(x)\nabla \phi(x)  + 2(\nabla \phi(x))^\top A(x) \xi + \xi A(x) \xi  = \sum_{k\in\,\Box^2_1} \left| (\nabla \phi(x)+\xi) \cdot \nabla^{\perp}\psi(x-k) \right|^2,
 \end{equation}
 we write \eqref{eq:turb_visc_upper_1} in the more compact form, for any $\xi\in \mbS^1$,
 \begin{equation}\label{eq:turb_visc_upper_2}
 	\nu_\kappa \leq \int_{\mbT^2} \left[\kappa |\nabla \phi(x)|^2 +\sum_{k\in\,\Box^2_1} \left| (\nabla \phi(x)+\xi) \cdot \nabla^{\perp}\psi_k(x-k) \right|^2  \right] \dd x
 \end{equation}
 Therefore, setting $\xi= e_1 =(1,0)$ and
 \begin{equation}
 	\phi(x) = \begin{cases}
 		 -x_1, & x_1 \in [0,c],\\
 		\frac{2c }{1-2c}x_1 - \frac{c}{1-2c}, &x_1 \in [c,1-c],\\
 		1-x_1, & x_1 \in [1-c,1].
 	\end{cases}
 \end{equation}
 for which,
 \begin{equation}
 	\nabla \phi(x) = \begin{cases}
 		-e_1, & x_1 \in (0,c)\cup (1-c,1)\\
 		\frac{2c}{1-2c}e_1, &x_1 \in (c,1-c)
 	\end{cases}
 \end{equation}
 we find the estimate
 \begin{align}
 	\nu_\kappa \, \leq  \,& \kappa \left( 2 \int_{0}^c \int_0^1 \dd x_1 \dd x_2 +  \int_{c}^{1-c}\int_0^1 \left|\frac{2c}{1-2c}\right|^2 \dd x_1 \dd x_2\right)\\
 	&+ \sum_{k\in \, \Box^2_1} \int_{B(k,c)}  \left| (-e_1+e_1) \cdot \nabla^{\perp}\psi(x-k) \right|^2 \dd x\\
 	= \,& 2\kappa \left(c +  \frac{4c^2}{(1-2c)^2}\right)\label{eq:c_leq_half_linear}
 \end{align}
 which concludes the proof for $c\in (0,\nicefrac{1}{2})$.

Since the estimate \eqref{eq:c_leq_half_linear} degenerates as $c\to \nicefrac{1}{2}$ we are required to adapt our argument in this case. To this end we define a family of approximate competitors. For $\delta \in (0,\nicefrac{1}{2})$ we set
 \begin{equation}
 	\phi^\delta(x)  = \begin{cases}
 		-x_1, & x_1 \in [0,\nicefrac{1}{2}-\delta],\\
 		\frac{1-2\delta}{2\delta}x_1 + \frac{2\delta - 1}{4\delta}, &x_1 \in [\nicefrac{1}{2}-\delta,\nicefrac{1}{2}+\delta],\\
 		1-x_1, & x_1 \in [\nicefrac{1}{2}+\delta,1].
 	\end{cases}
 	\end{equation}

 Therefore, by symmetry of the sum and radial symmetry $\psi$, for every $\delta \in (0,\nicefrac{1}{2})$, we have 
 \begin{align*}
 	\nu_\kappa \leq &\, \kappa \left( 2  \int_0^1  \int_{0}^{\nicefrac{1}{2}-\delta} \dd x_1 \dd x_2 +  \int_0^1  \int_{\nicefrac{1}{2}-\delta}^{\nicefrac{1}{2}+\delta} \left|\frac{1-4\delta}{8\delta}\right|^2 \dd x_1 \dd x_2\right)\\
 	&+ \sum_{k\in \, \Box^2_1} \int_{B(k,\nicefrac{1}{2}) \cap ([1/2-\delta,1/2+\delta]\times [0,1])}  \left| \left(\frac{1-4\delta}{8\delta}+1\right)e_1 \cdot \nabla^{\perp}\psi(x-k) \right|^2 \dd x \\
 	= \,& 2\kappa \left(\frac{1}{2}-\delta +  \frac{(1-4\delta)^2}{8\delta}\right)\\
 	& + 4 \int_{B(0,\nicefrac{1}{2})\cap ([\nicefrac{1}{2}-\delta,\nicefrac{1}{2}]\times[0,1])} \left| \left(\frac{1-4\delta}{8\delta}+1\right)e_1 \cdot \nabla^{\perp}\psi(x) \right|^2 \dd x.
 \end{align*}
Since we assume  $D^{a}\nabla^\perp \psi|_{\partial_{B(0,\nicefrac{1}{2})}} = 0$ for all $a\in \mbN^2$ and $\nabla^\perp \psi \in C^\infty(\mbT^2)$ we can Taylor expand to arbitrary order around any point $x\in \partial B(0,\nicefrac{1}{2})$. With the same $\delta \in (0,\nicefrac{1}{2})$ as above, let us define 
 \begin{equation*}
     B(0,\nicefrac{1}{2}) \cap \{x_1 = \nicefrac{1}{2}-\delta \} \eqqcolon  x_\delta =  \Big(\nicefrac{1}{2}-\delta, \sqrt{\delta(1-\delta)}\, \Big),
 \end{equation*}
for which 
\begin{equation*}
    |x_\delta - (\nicefrac{1}{2},0)| = \sqrt{\delta},
\end{equation*}
Hence, for any $x\in B(0,\nicefrac{1}{2})\cap ([\nicefrac{1}{2}-\delta,\nicefrac{1}{2}]\times [0,1])$ and $p\geq 0$, we obtain 
\begin{align}
 	|\nabla^\perp \psi(x) |=&\, \left| \nabla^\perp \psi(x) -\sum_{|a|\leq p} \frac{1}{a!}D^a \nabla^\perp \psi((\nicefrac{1}{2},0))(x-(\nicefrac{1}{2},0))^a \right|\\
    \leq &\,  \sup_{x \in B(0,\nicefrac{1}{2})\cap [\nicefrac{1}{2}-\delta,\nicefrac{1}{2}]\times [0,1]}|x - (\nicefrac{1}{2},0)|^{p+1}\max_{|a|=p+1}\|D^a \nabla^\perp \psi_1\|_{L^\infty(\mbT^2)}\\
    \leq &\, |x_\delta - (\nicefrac{1}{2},0)|^{p+1}\max_{|a|=p+1}\|D^a \nabla^\perp \psi\|_{L^\infty(\mbT^2)}\\
    =&\,\delta^{\frac{p+1}{2}} \max_{|a|=p+1}\|D^a \nabla^\perp \psi\|_{L^\infty(\mbT^2)}\\
    \lesssim &\, \delta^{\frac{p+1}{2}}\|\nabla^\perp \psi \|_{W^{p+1,\infty}(\mbT^2)} ,
 \end{align}
 Furthermore, one directly sees 
 \begin{equation}
 	\big|B(0,\nicefrac{1}{2})\cap [\nicefrac{1}{2}-\delta,\nicefrac{1}{2}]\times[0,1]\big| \leq  \delta \sqrt{\delta(1-\delta)} \lesssim \delta^{\nicefrac{3}{2}}.
 \end{equation}
 Putting all of this together, we find that for any $p\geq 1$,
 \begin{align*}
 	\int_{B(k_1,\nicefrac{1}{2})\cap [\delta,\nicefrac{1}{2}]\times[0,1]} \bigg| \left(\frac{1-4\delta}{8\delta}+1\right)e_1 \cdot &\nabla^{\perp}\psi(x)\,  \bigg|^2 \, \dd x\\
    &\, \leq \, \delta^{\frac{3}{2}}\left|\frac{1-4\delta}{8\delta}+1\right|^2 \sup_{x \in B(0,\nicefrac{1}{2})\cap ([\nicefrac{1}{2}-\delta,\nicefrac{1}{2}]\times [0,1])} |\nabla^\perp \psi(x)|^2\\
 	&\, \leq \,  \delta^{\frac{p+4}{2}}\left|\frac{1-4\delta}{8\delta}+1\right|^2 \|\nabla^\perp \psi\|^2_{W^{p+1,\infty}(\mbT^2)}.
 \end{align*}
 Hence, we find
 \begin{align*}
 	\nu_\kappa \, &\leq \,2\kappa \left(\frac{1}{2}-\delta +  \frac{(1-4\delta)^2}{8\delta}\right) + 2 \delta^{\frac{p+4}{2}}\left|\frac{1-4\delta}{8\delta}+1\right|^2 \|\nabla^\perp \psi\|^2_{W^{p+1,\infty}(\mbT^2)}\\
 	%
 	%
 	&\lesssim \left(1+\|\psi\|_{W^{p+1,\infty}}\right) \left(\kappa + \frac{\kappa}{\delta} + \delta^{\frac{p+4}{2}}\right).
 \end{align*}
 Now, let us set $\delta = \kappa^{\alpha}$ for some $\alpha>0$, to write
 \begin{equation*}
 		\nu_\kappa \lesssim_{\|\psi\|_{W^{p+1,\infty}}} \kappa + \kappa^{1-\alpha} +\kappa^{\frac{\alpha(p+4)}{2}}.
 \end{equation*}
 Since we are free to choose $\alpha>0$ we obtain the best result when
 \begin{equation*}
 	1-\alpha = \frac{\alpha(p+4)}{2} \quad \iff \quad \alpha = \frac{2}{p+6}.
 \end{equation*}
 Plugging this in we see that
 \begin{align*}
 		\nu_\kappa\lesssim_{\|\psi\|_{W^{p+1,\infty}}} \kappa + 2\kappa^{1-\frac{2}{p+6}}.
 \end{align*}
Hence, for any $n\geq \nicefrac{7}{2}$ there exists a $p\coloneqq p(n) \geq 1$ (in fact $p(n) = 2n- 6$) such that
\begin{align*}
 \nu_\kappa \lesssim_{\|\psi\|_{W^{p(n)+1,\infty}}}  \kappa + \kappa^{1-\frac{1}{n}}.
 \end{align*}
%
 %
 which concludes the proof.
\end{proof}
\begin{remark}
    Note that in the case $c=\nicefrac{1}{2}$ we crucially used the assumption that $\psi$ was smooth to take Taylor expansions of arbitrarily high order. However, since $\psi$ is compactly supported, it cannot be analytic and so we cannot take the limits $p(n),\, n \to +\infty$. This fact is reflected in our numerical experiments (see \autoref{fig:add_visc2}) where we observe that the exponent drops from being consistent with linear dependence on $\kappa$ for $c<\nicefrac{1}{2}$ to being strictly less than one for $c= \nicefrac{1}{2}$.
    %
\end{remark}
\subsubsection{The regime $\nicefrac{\sqrt{2}}{2}<c <\nicefrac{\sqrt{5}}{2}$:}
In this subsection we analyse the constant $C(c, \kappa)$ in the regime $\nicefrac{\sqrt{2}}{2}<c <\nicefrac{\sqrt{5}}{2}$. We prove that in this case $C(c, \kappa) \ge m >0$ for a constant $m$ independent of $\kappa$, thus this regime yields a strictly positive additional diffusivity $\nu_\kappa$, uniformly in $\kappa >0$.
\begin{definition}
    Let $\xi \in \mathbb{S}^1$ be given and define the affine space 
    $$\mathcal{V}_\xi=\left\{ F \in C^\infty(\T^2, \R^2): \int_{\T^2} F = \xi, \ \nabla^\perp \cdot F=0\right\}.$$
    We write $V_\xi$ for the closure of $\mathcal{V}_\xi$ in $L^2(\T^2, \R^2)$.
\end{definition}
 Let $c>1$, introduce on $V_\xi$ the quadratic functional
    \begin{equation*}
    \|F\|_A^2:=\int_{\T^2}F^\top(x)A(x)F(x)\dd x = \sum_{k\in\, \Box^2_1} \int_{\T^2}|F\cdot \nabla^\perp\psi(x-k)|^2 \dd x
    \end{equation*}
    \begin{lemma}\label{lem: inf equality}
        The infimum of $\|\cdot \|_A$ on $V_\xi$ coincides with the infimum on $\mathcal{V}_\xi$. 
    \end{lemma}
    \begin{proof}
        Since $\|\,\cdot\,\|_A$ is a convex functional, there exists a sequence $\{F^n\}_{n\in \mbN} \subseteq \mcV_\xi$ be a sequence such that $\lim_{n\to \infty} \|F^n\|_A = \inf_{F\in \mcV_\xi}\|F\|_A$, and for each $\delta>0,\, n\in \mbN$ let $F^{n, \delta}$ be a smooth approximation to $F^n$ such that $\|F^n- F^{n, \delta}\|_{L^2} = \delta$. This approximation can be constructed noticing that $F\in V_\xi$ implies $F= \xi + \nabla u$ for some $u\in \H^1$, and then approximating $u$. Since $\psi\in C^\infty(\mbT^2;\mbR)$ we have 
        \begin{align*}
            \|F^{n, \delta}\|_A &\le  \left(\sum_{k\in \,\Box^2_1} \int_{\T^2}|F^n\cdot \nabla^\perp\psi(x-k)|^2 + |(F^{n,\delta}- F^n)\cdot \nabla^\perp\psi(x-k)|^2\right)^{1/2}\\ 
            &\le \left(\sum_{k\in \,\Box^2_1} \|F^n\cdot \nabla^\perp\psi(x-k)\|_{L^2}^2 + \|(F^{n,\delta}- F^n)\cdot \nabla^\perp\psi(x-k)\|_{L^2}^2 \right)^{1/2} \\
            &\le \|F^n\|_A + 2\delta\|\nabla\psi\|_\infty \ ,
        \end{align*}
        where we used the subadditivity of the square root. Finally, taking the infimum over $n$ yields the result since $\delta$ was chosen arbitrarily. 
    \end{proof}
    From now on, we will keep $\xi$ fixed and drop the subscript $\xi$ in $V_\xi$ and $\mathcal{V}_\xi$. 
    For the next lemma we regard $F$ as a function defined on $\R^2$ and perform integrations over $[0,1]^2$.
\begin{lemma}\label{lem: line integral}
        Let $F\in \mathcal{V}$, $a\in [0,1]$ and define the curves 
        $$\gamma_H(t)= (t,a),\quad \gamma_V(t)= (a,t) $$
        and the functions
        $$\quad g_H(x)= x_2-a, \quad g_V(x)= -x_1+a.$$
       Then, it holds that 
       \begin{align*}
           \int_0^1F_1(t, a)dt = \int_{\gamma_H}g_H F d\tau =\xi_1 \quad \text{and}\quad  \int_0^1F_1(a, t)dt = \int_{\gamma_V}g_V F d\tau =-\xi_2.
       \end{align*}
\end{lemma}
\begin{proof}
    Consider the vector fields $v_1(x)=e_1, \ v_2(x)=e_2$. It is readily seen that $v_i(x)= \nabla^\perp g_i(x)$ with $i \in \{H,V\}$. For $F \in \mathcal{V}$ we have 
    $$\xi_i=\int _{[0,1]^2}F_i = \int_{[0,1]^2}F\cdot\nabla^\perp g_i= \int_{\partial [0,1]^2} g_iF \cdot d\tau $$
    For the sake of the computation fix $i=H$. Using the periodicity of $F$ and the symmetries of $g_H$ the integration reduces to 
    \begin{align*}\int_{\partial [0,1]^2} g_HF \cdot d\tau &= \int_0^1 x_2 F_2(1, x_2)\dd x_2 + \int_0^1 F_1(x_1, 1)\dd x_1 - \int_0^1 x_2 F_2(0, x_2)\dd x_2 \\
    &=\int_0^1 F_1(x_1, 1)\dd x_1
    \end{align*}
    to obtain the case for general $a\in [0,1]$, replace $F$ by $F(x_1, x_2-1+a)$. The case $i=V$ is analogous. 
\end{proof}
Thanks to this lemma we obtain the following corollary
\begin{corollary}\label{cor:c_great_root2_over_2}
Let $c>\nicefrac{\sqrt{2}}{2}$ and $\xi \in \mathbb{S}^1$. Then it holds that, 
\begin{enumerate}[label = \roman*)]
    \item  \label{it:norm_F_inf_over_V} $\inf_{F\in V} \|F\|_A >0.$ 
    \item \label{it:lower_bound_energy} There exists a strictly positive constant $m>0$ such that 
    \begin{equation}
       \inf_{u\in \mcH^1(\mbT^2)} \E_0(u)= \int (\xi+\nabla u(x))^tA(x)(\xi+\nabla u(x))\dd x \geq m
    \end{equation}
    \item \label{it:lower_bound_viscosity} For every $k>0$, if $u_k^*$ is the minimizer of $\E_k$, it holds that $\E_k(u_k^*)\ge \kappa + \kappa\|\nabla u_\kappa^*\|_{L^2_x} + m$ and thus 
    \begin{equation*}
        C(c, \kappa)\ge \kappa+ \kappa\|\nabla u_\kappa^*\|^2_{L^2}+m.
\end{equation*}
\end{enumerate}
\end{corollary}

\begin{proof}
    By Lemma \eqref{lem: inf equality}, it is enough to prove \ref{it:norm_F_inf_over_V} over $\mathcal V$ rather than $V$.
    Suppose by contradiction that there exists a sequence $F^n \in \mathcal{V}$ such that $\|F^n\|_A \rightarrow 0$. First we prove a statement regarding a special set of test functions, then we will show that if $c>\sqrt{2}/2$, this set of test functions is rich enough to derive a contradiction. Let $\varphi$ be an $L^2(\mbT^2;\mbR^2)$ function which can be written as
    \begin{equation}\label{test function}
        \varphi(x) := \sum_{k \in \, \Box^2_1}a_k(x)\nabla^\perp \psi(x-k). \end{equation}
    for some $a_k(x) \in L^2(\T^2)$ and where $\psi$ is as in \eqref{eq:psi_props}. We have
    $$\left| \int_{\T^2} F^n \cdot \varphi \right| =\left |\sum_{k\in \, \Box^2_1}\int_{\T^2} a_k(x)F^n \cdot \nabla^\perp \psi(x-k) \right|\le \sum_{k\in \Box^2_1} \|a_k\|_{L^2}\left( \int_{\T^2}|F^n \cdot \nabla^\perp \psi(x-k)|^2\right)^{1/2}.$$
    Which yields 
    \begin{equation}\label{weak conv zero}
        \left|\int_{\T^2} F^n \cdot \varphi \right| \lesssim \|F^n\|_A \xrightarrow[]{n} 0.
        \end{equation}
   Now choose any $\delta>0$ and a test function $\hat\varphi=(\varphi_1(x_2), 0)$, that is positive and equal to one if $|x_2 -\nicefrac{1}{2}| \le \nicefrac{\delta}{2}$ and zero if $|x_2 -\nicefrac{1}{2}| \ge\delta$.  By lemma \eqref{lem: line integral} and Fubini-Tonelli
    \begin{align*}
        \int_{\T^2}F^n \cdot \hat \varphi &=\int_{\nicefrac{1}{2}-\delta}^{\nicefrac{1}{2}+\delta}\int_0^1 F_1^n(x_1, x_2) \varphi_1(x_2)\dd x_1 \dd x_2 \\ &= \int_{\nicefrac{1}{2}-\delta}^{\nicefrac{1}{2}+\delta} \varphi_1(x_2)\left(\int_0^1 F_1^n(x_1,x_2) \dd x_1\right) \dd x_2
        \\ & =\int_{\nicefrac{1}{2}-\delta}^{\nicefrac{1}{2}+\delta}\varphi_1(x_2)\xi \dd x_2 
        \\ &\ge \frac{\delta}{2} \xi_1 \ .
    \end{align*}
      Assume now for the sake of the computations that $\xi_1\neq 0$. Then, this inequality, together with \eqref{weak conv zero} yields a contradiction, provided that we can write the test function $\hat \varphi$ in the form of \eqref{test function}. The last step of the proof is to show that this is possible provided that $c>\sqrt{2}/2$, and $\delta$ is sufficiently small. \\
      Recalling \autoref{lem:A_cases}, we know that, if $c>\sqrt{2}/2$, there exists a strip 
      $$S_c= [0,1]\times [1/2 - \delta(c), 1/2 + \delta(c)]$$
      for $\delta(c) = \sqrt{c^2 - \nicefrac{1}{4}} - \nicefrac{1}{2}  = >0$
      where the set of vectors $\B(x)=\{\nabla^\perp \psi(x-k)\}_{k \in \, \Box^2_1}$, contains at least two vectors. Moreover, if $x\in \mathring S_c$ (the interior of $S_c$), again by the proof of \autoref{lem:A_cases} , $\B(x)$ spans $\R^2$.
      In addition, thanks to Remark \eqref{rem:A_kernel}, we know that for $x \in\partial [0,1]^2$, $\B(x)$ spans the direction normal to the boundary, that is: 
     $$\operatorname{span}\B(x)  = \operatorname{span}(e_2) \qquad for \ x=(x_1, 0) \ or \ x=(x_1, 1) $$
    $$\operatorname{span}\B(x)  = \operatorname{span}(e_1) \qquad for \ x=(0, x_2) \ or \ x=(1, x_2).$$
    Consider the smaller strip 
    $$ S^{1/2}_c=[0,1]\times [\nicefrac{1}{2} - \nicefrac{\delta(c)}{2}, \nicefrac{1}{2} + \nicefrac{\delta(c)}{2}]$$
    The two facts above together ensure that the vector $e_1$ is in the range of $A(x)$ for every $x \in S^{1/2}_c$, thus, there exists functions $\tilde a_k(x)$ given by $\tilde a_k(x)= \nabla^\perp\psi(x-k)\cdot w(x)$ for $\omega(x)$ such that $A(x)\omega(x)= e_1$ such that $e_1= \sum_{k \in \, \Box^2_1} \tilde a_k(x)\nabla ^\perp \psi(x-k)$. Finally, thanks to the smoothness of $A(x)$ and the fact that $\omega(x)=e_2$ for $x=(0, x_2), \ (1,x_2)$, the coefficients $\tilde a_k(x)$ can be chosen bounded (even smooth) in the smaller strip $S^{1/2}_c$.  
    Now it is enough to chose $\delta \le \delta(c)/2$ in the definition of $\hat\varphi$, to see that we have the representation 
    $$\hat\varphi (x) = \sum_{k \in \, \Box^2_1 }\varphi_1(x_2)\tilde a_k(x) \nabla^\perp \psi(x-k)$$
    
    Finally, to conclude the proof of \ref{it:norm_F_inf_over_V}, if $\xi_1=0$, we repeat the same argument with $\varphi= (0, \varphi_2)$ with 
    $$\supp(\varphi) \subset [1/2 -\delta, 1/2+\delta]\times [0,1].$$
    Claim \ref{it:lower_bound_energy} follows straightforwardly by setting $F= \xi + \nabla u \in V$ while claim \ref{it:lower_bound_viscosity} follows simply by noting that 
    $$\E_\kappa(u)= \kappa + \kappa \|\nabla u\|_{L^2}^2 + \|\xi + \nabla u\|_A^2.$$
\end{proof}
\begin{remark}\label{rem:c_greater_root_2_over_2_scaling}
    Observe that the functional $\|\cdot \|_A$ is $1-$homogeneous in $\psi$, that is if we multiply the profile $\psi$ by a factor $\lambda$, and call the new matrix $A_\lambda(x)$ we get
     $$\|F\|_{A_\lambda{}}^2 = \lambda^2\sum_{k\in \, \Box^2_1} \int_{\T^2}|F\cdot \nabla^\perp\psi(x-k)|^2\dd x >\lambda^2 m.$$
     As a consequence, for $c\ge 1$, and $\kappa$ fixed, $C(c, \kappa)$ can be made arbitrarily large by scaling $\psi$ to $\lambda \psi$ in the definition of the patches (or equivalently choosing $\theta_k^N = \lambda/N$ in the scaling assumptions).
\end{remark}
\section{Numerical Simulations}\label{sec: numerics}
As anticipated in the previous sections, the behaviour of the total effective diffusivity $C(c,\kappa)$ as $\kappa \to 0$ is not completely understood. To address this issue, a numerical estimate of the total effective diffusivity $C(c,\kappa)$ can be performed by solving equation \eqref{eq:corrector}.
\subsection{Numerical setup}
We numerically solve the equation on a two-dimensional lattice $[0,1]\times[0,1]$, discretized on a $n\times n$ grid of grid-space $d = 1/ n$, and $H_\kappa(x) = \kappa I + A(x)$ is a $2n \times 2n$ matrix. 
In our numerical setup, the orthogonal gradient of $\psi$ is taken as:
\begin{equation} 
    \mbT^2 \ni x\mapsto \nabla^{\perp}\psi(x) = \frac{x^\perp}{|x| } \frac{\varphi(|x|)}{\|\varphi\|_{L^2_x}},
    \label{eq:nabla_psi_numerics}
\end{equation}
In order to be able to compare different values of the parameter $c$, the function $\varphi(x)$ is normalized, so that the following condition holds:
\begin{equation}
  \|\nabla \psi\|_{L^2_x} = 1 .
\end{equation}
The following choice of $\mbR_+\ni r \mapsto \varphi(r)\in \mbR$ is used, in order to satisfy conditions \eqref{eq:psi_props} :
\begin{equation} 
\varphi(r) = \frac{1}{r^2}  \exp{ \left( - \frac{a_1}{r^2} \right) } 
            \exp{ \left( - \frac{a_2 r^2}{ | r - c |} \right) }
            \label{eq:phi_numerics}
\end{equation}
with $a_1$, $a_2$ parameters. 

We define the second order elliptic operator $T$ in divergence form as:

\begin{equation}
    T \phi_e:=-\frac{1}{2}\left(\sum_{i=1}^2 D_i^{-}\left(\sum_{j=1}^2 H_{i j} D_j^{+} \phi_e \right)+\sum_{i=1}^2 D_i^{+}\left(\sum_{j=1}^2 H_{i j} D_j^{-} \phi_e \right)\right) \ , 
    \label{eq:matrix_t_def_numerics}
\end{equation}
where $D_i^{-}$ and $D_i^{+}$ are, respectively, the forward and backward difference operators. The two directions are denoted by the index $i=1,2$. This numerical approximation is consistent up to second order, see \cite{grossmann2007numerical}. In this framework, the corresponding divergence operator is defined as:
\begin{equation*}
    \text{div} f := \frac{1}{2} \sum_{i=1}^n (D_i^{-} + D_i^{+} ) f_i .
\end{equation*}
We solve the two linear systems
\begin{equation}
    T \phi_1 =   \text{div} \left( A \cdot e_1 \right) \ ; \quad 
     T \phi_2 =   \text{div} \left( A \cdot e_2 \right)
\label{eq:t_system_numerics}
\end{equation}
employing the direct solver implemented in MATLAB.

After solving the linear systems, we compute the total effective diffusivity, \eqref{eq:total_viscosity}, for different values of the parameter $c$. The total diffusivity is obtained from the real part of the first eigenvalue of the $2\times2$ matrix obtained by averaging on the whole grid. The computation of the total diffusivity from the other eigenvalue leads to consistent results, as the $2\times2$ matrix is (almost) diagonal. The compatibility of the imaginary part of both the eigenvalues with zero was checked. 

In all cases, the data obtained at different values of the molecular diffusivity $\kappa$ were computed for different values of the gridstep $d$: $d= 0.0100, 0.00222, 0.00167, 0.00125$. Then, the final estimate for the total diffusivity was obtained by extrapolating for $d \to 0$.

We check that a compatible result is obtained by plugging our solution into the definition of the total diffusivity in the variational setting, \eqref{eq:total_viscosity_var}. This second quantity is also extrapolated for $d \to 0$. 
\subsection{Results}\label{sec:numerical_results}
In the following, we report our results concerning the behaviour of the additional diffusivity, $C(c,\kappa) - \kappa$, as $\kappa \to 0$. In most of the section, we focus on a single choice of the profile function $\varphi(x)$; however, the study of $\varphi(x)$ for different choices of parameters is crucial to identify which of our results depends on the choice of the profile function and how. From the previous analysis, we expect our results to be stable under the choice of the profile function for $c \in (0, \nicefrac{1}{2})$. 
\subsubsection{Choice of profile function and corresponding solution of corrector equation}
As a first step, we investigate the dependence of the profile function $\varphi(x)$ on the parameters $a_1$ and $a_2$. We report our choice for $\varphi(x)$ in the upper plot of \autoref{fig:varphi}, while in the lower plots of \autoref{fig:varphi} we show other possible choices for the two parameters $a_1$ and $a_2$. The parameters have to be chosen carefully, as the function could drop dramatically to a value compatible with zero already at $x' < c$, inside the chosen support radius $c$. 
The matrix $T$ was de-singularized by subtracting to the corrector its mean value.
    \begin{figure}
    \begin{center}
       \includegraphics[width=.49\textwidth]{FIG/varphi.eps}  
    \end{center}
    \includegraphics[width=.49\textwidth]{FIG/varphi_a1.eps}
    \includegraphics[width=.49\textwidth]{FIG/varphi_a2.eps} 
    \caption{Up: plot of $\varphi(x) / \|\varphi\|_{L^2_x} $, with $\varphi(x)$ defined in \eqref{eq:phi_numerics}, with our choice $a_1 = 0.05$, $a_2 = 0.3$, at $c=1.2$. This choice of the parameter $c$ shows that for $c>1$ the profile function is not zero outside the interval [-1,1]. \\ 
    Down: different possible choices of the coefficients $a_1$ (left, $a_2$ is held fixed, $a_2 = 0.3$), $a_2$ (right, $a_1$ is held fixed, $a_1 = 0.05$), at $c=1.2$. Choices of parameters employed in numerical simulations are plotted in blue, with $a_1 = 0.05$, $a_2 = 0.3$. Note that the plot on the right hints that for $a_2 > 0.6$ the support of the function drops to a value compatible with zero already before $c=1.2$.}
\label{fig:varphi} 
\end{figure}

\subsubsection{Additional Diffusivity}\label{sec:aditional_diffusivity}
%
To study the dependence of the additional diffusivity on the parameter $c$ as $\kappa \to 0$, we start by performing simulations for a single set of parameters: $a_1 = 0.05$, $a_2 = 0.3$. Results obtained with this setup are reported in \autoref{fig:add_visc}: they are divided into four different groups, $c\in (0,\nicefrac{1}{2}]$, $c\in (\nicefrac{1}{2},\nicefrac{\sqrt{2}}{2})$,  $c\in (\nicefrac{\sqrt{2}}{2}), \nicefrac{\sqrt{5}}{2})$, $c > \nicefrac{\sqrt{5}}{2}$. From those data, we can already observe that for $c \in (0, \nicefrac{1}{2})$ the behaviour of the additional diffusivity for $\kappa \to 0$ is linear, with intercept indistinguishable from zero, while for values of $c$ in the interval $c \in (\nicefrac{1}{2}), \nicefrac{\sqrt{2}}{2}))$ the intercept is again compatible with zero, but the additional diffusivity seems to follow a power law behaviour. The same behaviour, but with non-zero intercept, is present for $c \in (\nicefrac{\sqrt{2}}{2}) ,\nicefrac{\sqrt{5}}{2})$. On the contrary, for $c > \nicefrac{\sqrt{5}}{2}$ the intercept seems to be still non-zero, however the behaviour of the additional diffusivity is again linear. 
We remark that these result are in perfect agreement with the conclusion of \autoref{th:main_introduction}, moreover, we have conclusive results, from the numerical simulations, also in the regime $c\in (\nicefrac{1}{2}), \nicefrac{\sqrt{2}}{2})$, which was outside the reach of our theorem. The case $c=\nicefrac{1}{2}$ is in-between the two regimes, the power law behaviour and the linear behaviour, probably due to numerical artifacts.\\

To better investigate the different behaviours of the additional diffusivity for different values of the parameter $c$, we perform a set of fits in the range $0.0001 < \kappa < \kappa^\ast$ using the following function:
\begin{equation*}
    f(\kappa) = a \kappa^n + q .
\end{equation*}
We choose different values of $\kappa^\ast$, and report results for $n$ and $q$ in \autoref{fig:inter}. The fit procedure was performed by taking into account constraints for the parameter $a \in [0, + \infty)$, as values outside this interval would be unphysical and so are not considered. Those constraints were enforced by fitting the parameter $a'$, with $a = e^{a'}$. The same procedure was not applied for the parameter $q$, to allow it to become negative if very small.
The stability of the fit procedure was tested by performing jackknife procedure for each dataset in the following way: first, the fit was performed for different values of $\kappa^\ast$ in the interval $(0, \kappa^\ast]$, with $\kappa^\ast \in [0.0016, 0.005]$; then, the jackknife procedure was applied to the fit performed in each interval. To ensure the same effect on each dataset, we removed $95\%$ of the points in each interval, not only one. 
The jackknife estimate for each interval was used to to reduce bias and estimate uncertainty, in the usual fashion.
 Typically, in the case of the power law exponent and the coefficient $a$, the fit remains stable only for small $k^\ast$, however to estimate the final value also strictly larger $k^\ast$ estimates were taken into account. In particular, the power law exponent has a reduced stability interval for $c>1.1$. The intercept value, on the other hand, is stable for larger values of $k^\ast$. 
\begin{figure}
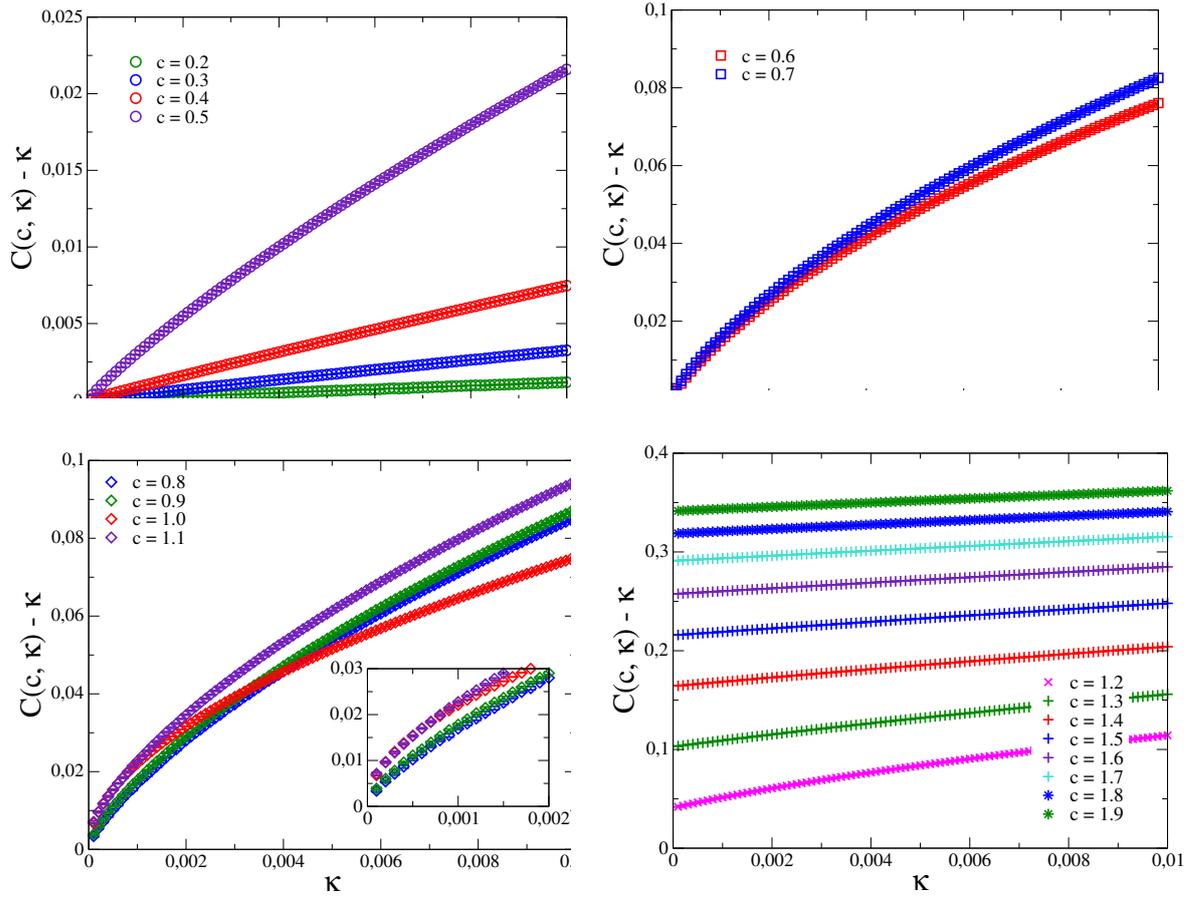

    \begin{center}
      \setlength{\unitlength}{\textwidth}
        \includegraphics[width=.49\textwidth]{FIG/add_visc_small_c.eps}
        \includegraphics[width=.49\textwidth]{FIG/add_visc_medium_c.eps} \\
        \includegraphics[width=.49\textwidth]{FIG/add_visc_medium_2_c.eps}
         \includegraphics[width=.49\textwidth]{FIG/add_visc_large_c.eps}
            \end{center}\vspace{-1em}
    \caption{
    Plot of additional diffusivity $C(c, \kappa)$, for different values of $c$, from left to right and top to bottom: $c\in (0, \nicefrac{1}{2}]$,$c\in(\nicefrac{1}{2}, \nicefrac{\sqrt{2}}{2})$, $c \in (\nicefrac{\sqrt{2}}{2}, \nicefrac{\sqrt{5}}{2})$, $c> \nicefrac{\sqrt{5}}{2}$ (with zoom at $\kappa \to 0$ if needed). The additional diffusivity is computed by means of \eqref{eq:total_viscosity}. Compatibility with results obtained by using \eqref{eq:total_viscosity_var} was separately checked.  The profile function used is $\varphi(x)$, defined in \eqref{eq:phi_numerics}, with $a_1 = 0.05$, $a_2 = 0.3$. The top right panel exhibits values of $c$ in the range $c\in (\nicefrac{1}{2}, \nicefrac{\sqrt{2}}{2})$, which we highlight lies outside the statements of \autoref{th:main_introduction}.}
\label{fig:add_visc2} \end{figure}
\begin{figure}
    \begin{center}
      \setlength{\unitlength}{\textwidth}
        \includegraphics[width=.49\textwidth]{FIG/inter.eps}
        \includegraphics[width=.49\textwidth]{FIG/exp.eps}
            \end{center}\vspace{-1em}
    \caption{ Intercept estimate (left) and power law exponent estimate (right) as a function of the parameter $c$. The profile function used is $\varphi(x)$, defined in \eqref{eq:phi_numerics}, with $a_1 = 0.05$, $a_2 = 0.3$. } 
\label{fig:inter} \end{figure}
As it can be seen from \autoref{fig:inter}, for $c \in (0, \nicefrac{1}{2})$ and $c > \nicefrac{\sqrt{5}}{2}$, while we do not obtain exactly $n=1$ for our estimate, the power law exponent approaches the value 1.0 as the fit interval restricts to the values $\kappa \to 0$. In particular, as noted above, the value of $c=\nicefrac{1}{2}$ seems to correspond to a transition behaviour for the additional diffusivity, probably as a consequences of numerical artifacts. Moreover, for $c\in (0,\nicefrac{1}{2}]$ and $c\in (\nicefrac{1}{2},\nicefrac{\sqrt{2}}{2})$ the intercept estimate is compatible with zero. In the interval $c\in (\nicefrac{\sqrt{2}}{2}, \nicefrac{\sqrt{5}}{2})$, the intercept estimate does not seems to be compatible with zero for $c = 1.0$ and $c=1.1$, while for $c=0.8$ the non-zero value is too small to draw conclusions. For $c > \nicefrac{\sqrt{5}}{2}$, the intercept estimate grows as a function of the parameter $c$.
\subsubsection{Sensitivity with Respect to Profile Parameters} 

We close this section by investigating the sensitivity of our results to the choice of the two parameters of the profile function $\varphi(x)$, $a_1$ and $a_2$. 
For specific choices of the parameter $c$, one for each of the investigated regimes, we compute the additional diffusivity for three different choices of $a_1$ and $a_2$, $a_1 = 0.05, 0.5, 1.0$ and $a_2 = 0.08, 0.3, 0.6$. When varying the parameters, one has to take into account that, if the function $\varphi(x)$ goes to zero too fast for $x \to c$, there will be numerical artifacts which have the effect of `shrinking' the support radius (see \autoref{fig:varphi}). 
To avoid the introduction of an `effective support', the values were chosen in order to keep the numerical support (i.e. the value of $x$ for which the function $\varphi$ becomes smaller than a tolerance value) close to $c$. However, even in this case, we cannot expect complete agreement between the different parameter choices.
\begin{figure}
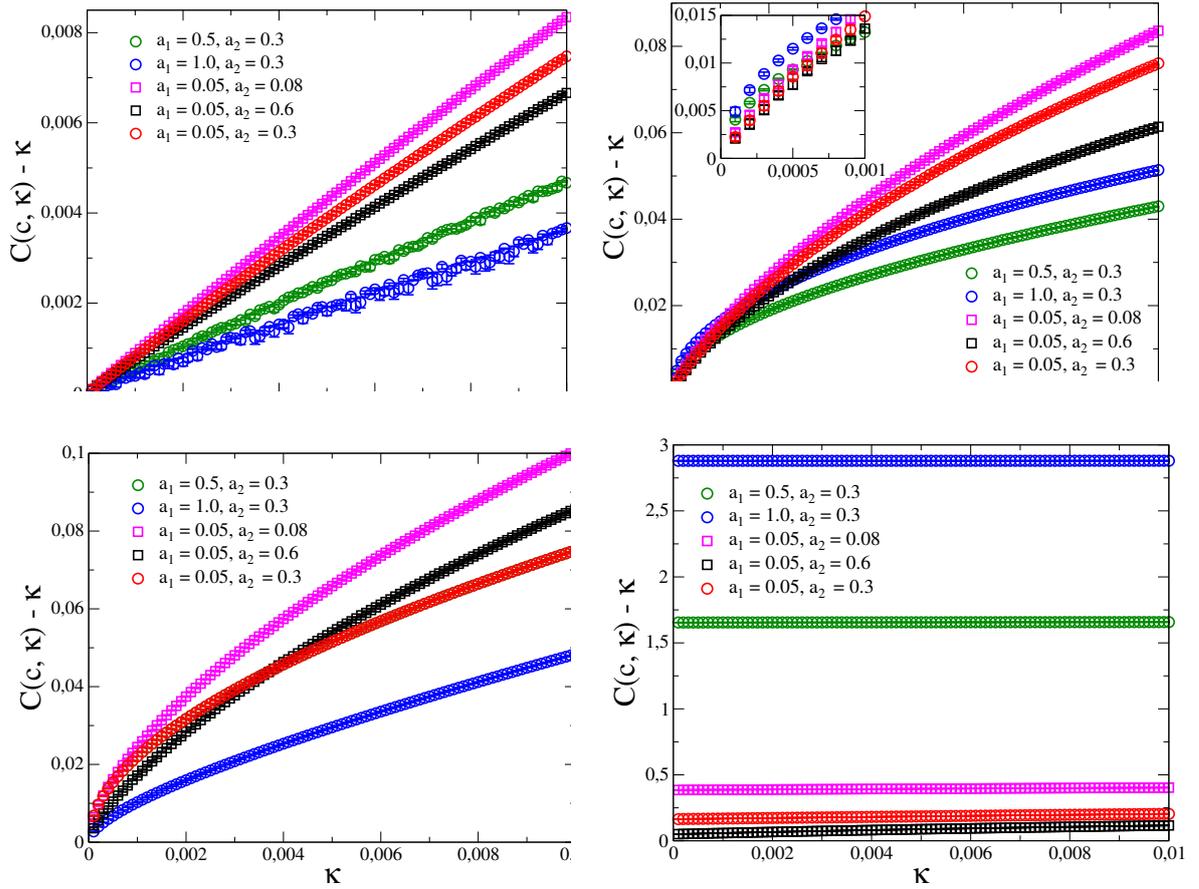

    \begin{center}
        \setlength{\unitlength}{\textwidth}
        \includegraphics[width=.49\textwidth]{FIG/add_visc_a1_a2_0.4.eps}
         \includegraphics[width=.49\textwidth]{FIG/add_visc_a1_a2_0.6.eps} \\
          \includegraphics[width=.49\textwidth]{FIG/add_visc_a1_a2_1.0.eps}
           \includegraphics[width=.49\textwidth]{FIG/add_visc_a1_a2_1.4.eps}
    \end{center}\vspace{-1em}
    \caption{Plot of additional diffusivity, for the values of the parameter $c = 0.4$ (up, left), $c = 0.6$ (up, right), $c = 1.0$ (down, left), $c = 1.4$ (down, right), computed as in \eqref{eq:total_viscosity}. Compatibility with results obtained by using \eqref{eq:total_viscosity_var} was separately checked. The profile function used is $\varphi(x)$, defined in \eqref{eq:phi_numerics}, with $a_1 = 0.05, 0.5, 1.0$, $a_2 = 0.3$, and $a_1 = 0.05$, $a_2 = 0.08, 0.3, 0.6$.}
\label{fig:add_visc_a1_a2} \end{figure}
The results obtained for the additional diffusivity for the different choices of the parameters $a_1$ and $a_2$ are reported in \autoref{fig:add_visc_a1_a2}. The estimates for the power law exponent and the intercept for the different choices of the parameter $a_1$ are reported in \autoref{fig:inter_a1}, while the results for $a_2$ are reported in \autoref{fig:inter_a2}.
\begin{figure}
    \begin{center}
      \setlength{\unitlength}{\textwidth}
        \includegraphics[width=.49\textwidth]{FIG/inter_a1.eps}
        \includegraphics[width=.49\textwidth]{FIG/exp_a1.eps}
            \end{center}
    \caption{Intercept estimate (left) and power law exponent estimate (right) as a function of the parameter $c$. The profile function used is $\varphi(x)$, defined in \eqref{eq:phi_numerics}, with $a_1 = 0.05, 0.5, 1.0$, $a_2 = 0.3$. }
\label{fig:inter_a1} \end{figure}
\begin{figure}[H]
    \begin{center}
      \setlength{\unitlength}{\textwidth}
        \includegraphics[width=.49\textwidth]{FIG/inter_a2.eps}
        \includegraphics[width=.49\textwidth]{FIG/exp_a2.eps}
            \end{center}
    \caption{Intercept estimate (left) and power law exponent estimate (right) as a function of the parameter $c$. The profile function used is $\varphi(x)$, defined in \eqref{eq:phi_numerics}, with $a_1 = 0.05$, $a_2 = 0.3, 0.08, 0.6$. }
\label{fig:inter_a2} \end{figure}
From data in \autoref{fig:inter_a1}, it is evident that for $c \in (0, \nicefrac{1}{2})$ there is not much difference in the (almost linear) behaviour of the additional diffusivity for different values of the parameter $a_1$; the same observation can be drawn from \autoref{fig:inter_a2} for the case of the varying parameter $a_2$. However, the results for $n$ outside of this range seem to depend on the specific parameter choice, while the zero-valued intercept in the range $c \in (\nicefrac{1}{2}, \nicefrac{\sqrt{2}}{2})$ remains unchanged (see \autoref{fig:inter_a1} and \autoref{fig:inter_a2}) and there seems not to be influence of the choice of the parameters $a_1, a_2$ on it.

\appendix
\section{Symmetries of the Homogenized Matrix} \label{app_A1}
We prove \autoref{lem:homogenised_upper_lower_symmetric} and \autoref{prop: diagonality}. 
\subsection{Proof of \autoref{lem:homogenised_upper_lower_symmetric}} \label{sec:standard_homogenised_statements}
\begin{proof}[Proof of \ref{lem:homogenised_upper_lower_symmetric}] 
To prove \ref{it:bar_H_kappa_upper_bnd} we use,  the uniform upper bound of $H_\kappa$, Cauchy-Schwarz and and Jensen inequality to give
        \begin{align*}
            |\bar H_\kappa\xi| & =\left|\sum_{j=1,2} \int_{\T^2} \xi_i H_\kappa(e_i + \nabla \phi_i) \right| \le \Lambda \sum_{j=1,2}  \int_{\T^2}  |\xi_i|\left|e_i + \nabla \phi_i  \right| 
            \\&\le \Lambda |\xi| \left(\sum_{j} \left|\int_{\T^2}  |e_i + \nabla \phi_i |\right|^2\right)^{1/2}  \le \Lambda |\xi| \left(\sum_{j=1,2}\int_{\T^2}|e_i + \nabla \phi|^2\right)^{1/2}
        \end{align*}
To prove \ref{it:bar_H_kappa_lower_bnd}, note that the corrector $\phi_\xi = \xi_i \phi_{i}$ satisfies 
        \begin{equation*}
            \operatorname{div}(A(\xi + \nabla \phi_\xi))=0
        \end{equation*}
        thus 
        \begin{align*}
        |\xi \cdot \bar H_k \xi|&= \int_{\T^2}\xi \cdot H_{\kappa}(\xi + \nabla \phi_\xi) =  \int_{\T^2}(\xi + \nabla \phi_\xi) \cdot H_{\kappa}(\xi + \nabla \phi_\xi) \\
        &\ge \lambda \int_{\T^2}|\xi + \nabla \phi_i|^2 = \lambda (|\xi|^2 + \|\nabla \phi_\xi\|^2_{L^2})
    \ge \lambda |\xi|^2.
        \end{align*}
Finally, \ref{it:bar_H_kappa_symmetric} follow from the symmetry of $H_\kappa$, inherited from that of $A^1$, namely,
\begin{align*}
            (\bar{H}_{\kappa})_{ij} &= e_j \cdot \int_{\T^2}H_\kappa(e_i + \nabla \phi_i) =  \int_{\T^2}(e_j + \nabla \phi_j)\cdot H_\kappa(e_i + \nabla \phi_i) \\
            &= \int_{\T^2}(e_i + \nabla \phi_i)\cdot H_\kappa(e_j + \nabla \phi_j) = (\bar{H}_{\kappa})_{ji}.
        \end{align*}
\end{proof}
The proof of the fourth point, the diagonality of $\bar H_\kappa$, is more delicate and require some preliminary lemmas, which we collect in the next subsection. 
\subsection{Diagonality of the Homogenized Matrix} \label{sec:homogenised_diagonallity}
We want to prove
\begin{proposition}[ \autoref{prop: diagonality}]\label{prop_diag}
   For the choice $M= H_\kappa=\kappa I + A$, we have $\bar H_k = C(c, \kappa) I$, for some constant $C(c, \kappa)\ge \kappa$.
\end{proposition}
The proof of this proposition relies on the symmetries of $x \rightarrow A^N(x)$, which are inherited by $\bar H_\kappa$. The matrix field $A^N(x)$ is not isotropic, but it is invariant under transformation of the lattice $(\nicefrac{1}{N}\Z^2)$. As a consequence, as $N \rightarrow \infty$, $H_\kappa$ is `\textit{asymptotically} isotropic', thus it is not a surprise that $\overline{H_k}$ is a multiple of the identity. The `nonlinear' nature of the homogenization is hidden in the constant $C(c, \kappa)$ which is a non-trivial function of the parameters. The fact that $C(c, \kappa)\ge \kappa$ is the content of point $ii)$ in \autoref{lem:homogenised_upper_lower_symmetric}. 
In order to prove \autoref{prop: diagonality} we need preliminary lemmas. 
Introduce the matrix $$J= \begin{pmatrix}
    0 & 1 \\-1 & 0
\end{pmatrix}$$
The matrix $J$ satisfies $JJ^t=J^tJ = - JJ=  I$, and for any $v\in \R^2, B\in M_{2\times 2}$, $v^\perp =Jv$, and
$$(Bv)^\perp = JBv=  J B J^t (v^\perp).$$
Given a matrix $B \in M_{2\times 2}$ define $B^J = JBJ^t$. Notice that if $B$ is orthogonal, i.e. $B^tB = BB^t= I$ then, also $B^J$ is, since $(B^J)^tB^J = JB^t J^tJBJ^t = I$.
\begin{lemma}\label{lemma: symmetry of A}
    Let $ \tilde R:\T^2 \rightarrow \T^2$ be such that $\tilde Rx =  R(x-j) + j $ where $j=(1/2, 1/2)$ and $R$ is a linear isometry of the half lattice $\Z^2$. Then it holds 
    $$A( \tilde Rx) = R^JA(x)(R^J)^t,$$
   and analogously
   $$A(\tilde{R}^{-1}x)=(R^J)^tA(x)R^J .$$
\end{lemma}
\begin{proof}
Noticing that for $k \in \Z^2$, there exists a unique $ k'\in \Z^2$ such that $k=  \tilde Rk '$, we write  
\begin{align*}
    A( Rx)&= \sum_{k\in \Z^2}\nabla^\perp \psi( \tilde Rx -k)\otimes \nabla^\perp \psi( \tilde Rx -k) \\
    &= \sum_{k\in \Z^2}\nabla^\perp \psi( \tilde R x-\tilde Rk')\otimes \nabla^\perp \psi( \tilde R x-\tilde Rk')\\
    &= \sum_{k\in \Z^2}\nabla^\perp \psi( R (x-k'))\otimes \nabla^\perp \psi( R(x-k'))
\end{align*}
Now since $\nabla ^\perp \psi(|x|)= \frac{x^\perp}{x}f'(|x|)$, and  $|Rx|=|x|$, we have
$$\nabla^\perp \psi(R (x-k')) = \frac{\left(R(x -k')\right)^\perp}{|x-k'|}f'(|x-k'|).$$
By the properties of the matrix $J$, we get 
$$\nabla^\perp \psi(R (x -k'))= JRJ^t \left( \nabla^\perp \psi(x- k')\right).$$
Now we observe that for $v \in \R^2$ and $B\in M_{2\times 2}$ 
$$(Bv \otimes Bv)_{lm} = B_{lj}v_j B_{mi}v_i = B_{lj}(v_jv_i B^t_{im}) = B_{lj}(v\otimes v B^t)_{jm} = (B (v\otimes v) B^t)_{lm}$$
Putting all together we obtain the desired identity. 
\end{proof}
\begin{remark}
    Observe that, since $R$ is a linear isometry, then $R^{-1}=R^t$ and the following identity holds 
    $$JRJ^t \left( (\nabla^\perp \psi)(\tilde R^{-1}x)\right)= \nabla ^\perp (\psi ( \tilde R^{-1}(x) ))$$
    Moreover, the set $G=:\left\{ R: \R^2 \mapsto \R^2 :  R \text{ is a linear isometry of }  \Z^2 \right\}$ is a group under usual matrix multiplication and it is generated by reflections and $\pi/4$-rotations $$r = \begin{pmatrix}
        1 & 0 \\ 0& -1
    \end{pmatrix} \qquad J= \begin{pmatrix}
    0 & 1 \\-1 & 0
\end{pmatrix}.$$
    Finally, it holds $J^J = J, \ r^J=-r$. As a consequence, for every $R\in G$, it holds $R^J= (\det R)R$.
\end{remark}
\begin{proposition}
    Let $\xi \in \R^2$, let $\phi_\xi =\sum_{i=1,2}\xi_i\phi_i$  where $\phi_i$ is the solution to the cell problem \eqref{eq:corrector}, and $R$ be as above. Then it holds $\phi_{R^J\xi} = (\det R) \phi_\xi(\tilde R^{-1}x)$, in particular
    \begin{equation}\label{lem: symm corrector}
        \nabla \phi_{R^J\xi}(y)= R^J\nabla \phi_\xi (\tilde R^{-1}y)
        \end{equation}
\end{proposition}
\begin{proof}
First, recall that $\phi_\xi$ solves 
$$\nabla \cdot (\kappa I + A)(\xi + \nabla \phi_\xi)=0. $$
    We use the weak formulation of the cell problem \eqref{eq:corrector}. We must show that for every $F\in \H^1({\T^2})$ we have 
    \begin{equation*}
       \kappa\int_{\T^2} \brak{\nabla F(x), R^J v + R^J\nabla \phi_v(\tilde R^{-1}x)}\dd x + \int_{\T^2} \brak{\nabla F(x), A(x)[R^J v + R^J\nabla \phi_v(\tilde R^{-1}x)]  }\dd x = \RN{1}+ \RN{2} =
       0
    \end{equation*}
    First we use the symmetries of $A$ from \autoref{lemma: symmetry of A} to write 
    \begin{align*}
        \RN{2} &= \int_{\T^2} \brak{\nabla F(x), [A(x)R^J] (\xi + \nabla \phi_\xi(\tilde R^{-1}x))  }\dd x \\
        &=\int_{\T^2} \brak{\nabla F(x), R^JA(\tilde R^{-1}x) (\xi + \nabla \phi_\xi(\tilde R^{-1}x))  }\dd x  \\
        &= \int_{\T^2} \brak{ (R^J)^t\nabla F(x), A(\tilde R^{-1}x) (\xi + \nabla \phi_\xi(\tilde R^{-1}x))  }\dd x. 
    \end{align*}
    Now we perform the change of variable $y = \tilde R^{-1}x$ and we notice that $(R^J)^t\nabla F(\tilde R y)= (\det R)\nabla (F(\tilde R y)) $. Calling $G= F\circ \tilde R$ thus we get 
    $$\RN{2} = \det R \int_{\T^2} \brak{\nabla G(y), A(y)(\xi + \nabla \phi_\xi(y))} dy.$$
    Analogously 
    $$\RN{1}= \det R \int_{\T^2}\brak{ \nabla G(y), \kappa(\xi + \nabla \phi_\xi)}$$
    Putting this together and using that $\phi_\xi$ solves the cell problem for $\xi$, we get $\RN{1} + \RN{2} = 0$ as desired. By the uniqueness of the solution to the cell problem, we get the equality $\phi_{R^J\xi} = (\det R) \phi_\xi(\tilde R^{-1}x)$.  
\end{proof}
We are now ready to prove the main proposition of this section
\begin{proof}[Proof of \autoref{prop_diag}]
We show that, for every $R \in G$ as above, 
\begin{equation}\label{hom matrix symmetry}
    \bar H_\kappa R^Jv = R^J\bar H_\kappa v.
\end{equation}
This is equivalent to $(R^J)^t \bar H_\kappa R^J = \bar H_\kappa$, which implies that $\bar H_\kappa$ is a multiple of the identity. Indeed, the operation $R \rightarrow R^J$ is a group isomorphism of $G$, thus we get $R^t\bar H_\kappa R = \bar H_\kappa$ for every $R\in G$, Now using $R=r$ gives that the off diagonal terms of $\bar H_\kappa$ are zero, while using $R=J$ gives that the diagonal terms are equal. 
We now prove \eqref{hom matrix symmetry}. Recalling that since $R$ was assumed to be a linear isometry of the lattice, it is an orthogonal matrix so that $R^J(R^J)^t=I$ and hence
\begin{align*}
 \overline H_\kappa  R^J\xi &= \int \left(\kappa I +A(x))(R^J\xi + \nabla\phi_{R^J\xi}(x)\right) \dd x \\
 & =  \kappa \int_{\T^2}(R^J\xi + \nabla\phi_{R^J\xi}(x))\dd x + \int R^J \left((R^J)^tA(x)R^J\xi + (R^J)^tA(x) \nabla \phi_{R^J\xi}(x)\right)\,\dd x,
 \end{align*}
So plugging in the identity \eqref{lem: symm corrector}, using the symmetry of $A$ from \autoref{lemma: symmetry of A}, changing variables in the integral and using again the fact that $\tilde{R}$ is orthogonal, we obtain
\begin{align*}
     \overline H_\kappa  R^J\xi &= R^J\left(\kappa \int_{\T^2}(\xi + \nabla\phi_\xi (\tilde R^{-1}x))\dd x + \int \big((R^J)^tA(x)R^J\xi + (R^J)^t A(x) R^J\nabla \phi_{\xi}(\tilde R^{-1}x)\big) \, \dd x \right) \\
     &= R^J\left(\kappa \int_{\T^2}(\xi + \nabla\phi_\xi (\tilde R^{-1}x))\dd x + \int A(\tilde R^{-1}x)\xi + A(\tilde R^{-1}x)\nabla \phi_{\xi}(\tilde R^{-1}x))\right) \\
     &= R^J\left(\kappa \int_{\T^2}(\xi + \nabla\phi_\xi (y))dy + \int A(y)\xi + A(y)\nabla \phi_{\xi}(y))dy\right)\\
     &= R^J\left(\int_{\T^2} (\kappa I + A(y))(\xi + \nabla \phi_\xi(y))dy\right) \\
     &= R^J \bar H_\kappa \xi
\end{align*}
which concludes the proof. 
   \end{proof}

\section{Proof of the Geometric Lemma}\label{A1_b}
In this section we provide a proof of the geometric \autoref{lem:geometry_fact}
\begin{proof}[Proof of \autoref{lem:geometry_fact}]
Given a point $x\in \mbT^2$, a unit vector $\eta \in \mbR^2$ and an angle $\theta \in [0,\pi)$ let us introduce the notation 
\begin{equation*}
    C^x_{\eta, \theta}:= \left\{\xi:\  |\xi\cdot (x-\eta) |\le |\xi|\cos \theta\right\}
\end{equation*}
to denote the cone with normal $\eta$ and opening angle $\theta$. We note given $x\in \mbT^2$, the set of vectors $v\in \mbR^2\setminus\{0\}$ such that \eqref{eq:ellipticity_geometry_fact} does not hold for a given $k \in \{(0,0),\,(1,0),\,(0,1),\,(1,1)\}$ is exactly the set $C^x_{(x-k)^\perp,\frac{\pi}{4}}$. Hence, it is enough to show that for any $x\in \mbT^2$
\begin{equation}\label{eq:geometry_intersection_claim}
    \bigcap_{k} C^x_{(x-k)^\perp,\frac{\pi}{4}} = \{x\},
\end{equation}
so that in turn there exists at least one vector $v\in \partial B(0,1)$ lying outside of at least one of the cones $C^x_{(x-k)^\perp,\frac{\pi}{4}}$. In this case \eqref{eq:ellipticity_geometry_fact} holds for this $v$ and this $k$.

Its an easy exercise to show that fixing $\theta \in [0,\pi)$ and any $x\in \mbT^2$
\begin{equation*}
    C^x_{\eta_1, \theta} \cap C^x_{\eta_2, \theta} = \left\{x\right\} \quad \iff \quad |\arccos(\eta_1\cdot \eta_2)|\geq \theta. 
\end{equation*}
That is, the two cones have non-trivial intersection if and only if the angle between the normals $\eta_1$ and $\eta_2$ is wider than $\theta$ and smaller than $\pi-\theta$ (up to multiples of $\pi$).

Let $k$ and $k'$ be adjacent members of the set $\{k_i\}_{i=1}^4 \coloneqq \{(0,0),\,(0,1),\,(1,0),\,(1,1)\}$. Denoting by $\angle(k,x,k')$ the angle (at $x$) made by the points $k,x k'$ then one readily checks that
\begin{equation}
    \min_{x\in \mbT^2} |\angle(k,x,k')| \geq \frac{\pi}{4},
\end{equation}
where the extreme case obtains for $x \in \{(0,0),\,(0,1),\,(1,0),\,(1,1)\} \setminus \{k,\,k'\}$. Since angles are preserved by translations and rotations, it must also hold that
\begin{equation}\label{eq:x_centred_lower_bound}
     \min_{x\in \mbT^2} \angle\big((x-k)^\perp,x,(x-k')^\perp\big) \geq \frac{\pi}{4},
\end{equation}
where we now assume that $k$ and $k'$ are such that $(x-k)^\perp$ and $(x-k')^\perp$ are adjacent. That is, there is no other $\tilde{k}\in \{(0,0),\,(0,1),\,(1,0),\,(1,1)\}$ such that 
\begin{equation*}
    \angle\big((x-k)^\perp,x,(x-\tilde{k})^\perp\big)  \wedge  \angle\big((x-\tilde{k})^\perp,x,(x-k')^\perp\big)  \leq \angle\big((x-k)^\perp,x,(x-k')^\perp\big). 
\end{equation*}
On the other hand, for all $x\in \mbT^2$, one has 
\begin{equation}
    \sum_{i=1}^4 \angle\big((x-k_i)^\perp,x,(x-k_{i+1 \text{ mod } 4})^\perp\big) = 2\pi, 
\end{equation}
Hence, for any $x\in \mbT^2$, there must exist at least one $i\in \{1,\ldots,4\}$ such that
\begin{equation}
  \angle\big((x-k_i)^\perp,x,(x-k_{i+1})^\perp\big) < \frac{\pi}{2},
\end{equation}
since if this were not the case (i.e. all angels were greater than $\nicefrac{\pi}{2}$ simultaneously) the sum would be greater than $2\pi$.

Hence, for any $x\in \mbT^2$, there exists at least one $i\in \{1,\ldots,4\}$ (understood with periodicity) such that
\begin{equation}
    \frac{\pi}{4} \leq  \angle\big((x-k_i)^\perp,x,(x-k_{i+1})^\perp\big) <\frac{\pi}{2}.
\end{equation}
Hence, for the same $x\in \mbT^2$ and $i \in \{1,\ldots,4\}$ as above
\begin{equation*}
    C^x_{(x-k_i)^\perp,\frac{\pi}{4}} \cap C^x_{(x-k_{i+1})^\perp,\frac{\pi}{4}} = \{x\}.
\end{equation*}
So that \eqref{eq:geometry_intersection_claim} holds. To prove the last statement, assume that the selected $\bar k$ is such that $|x-\bar k|> \sqrt{5}/2$. Without loss of generality, suppose $\bar k = (1,1)$. Then $x$ belong to the region $[1/2, 1/2]^2$. Consider the angle $\angle (1,0), x, (0,1)$. It is easy to see that, by construction, this angle is at most $\pi$ (it is exactly $\pi$ when $x=(1/2, 1/2)$), and at least $\pi/2$ (when $x = (0,0)$). This implies that the remaining angles $\angle (0,0), x, (0,1)$ and $\angle (0,0), x, (1,0)$ cannot be both strictly larger than $\pi/4$, since otherwise, the sum of the three angles would be larger than $2\pi$, which cannot be. This implies that at least one among the cones relative to $(1,0)$ and $(0,1)$ must be disjoint from the cone relative to $(0,0)$. The proof is completed by noticing that for each $x$ such that $|x-(1,1)|>\sqrt{5}/2$, we have $|x-k|< \sqrt{5}/2$ for any $k=(0,0), (1,0), (0,1)$.

\end{proof}

\section*{Acknowledgment}
S.M. acknowledges partial support by Italian Ministry of University and Research under the PRIN project Convergence and Stability of Reaction and Interaction Network Dynamics (ConStRAINeD) (2022XRWY7W, CUP I53D23002460001) and of INdAM through the INdAM-GNAMPA Project Modelli stocastici in Fluidodinamica e Turbolenza (CUP \#E5324001950001\#).

A.M. acknowledges support from the ERC Advanced Grant no. 101053472 for hosting them on a long term visit to Scuola Normale Superiore di Pisa during the production of this work.

In addition all authors wish to thank F. Flandoli, E. Luongo, F. Grotto, C. Bonati, L. Robol and M. Viviani for insightful discussions during development of these results.

\bibliography{biblio.bib,Mayorcas}{}

\newcommand{\etalchar}[1]{$^{#1}$}
\begin{thebibliography}{CGH{\etalchar{+}}03}

\bibitem[Agr24a]{agresti_global_2024}
A.~Agresti.
\newblock Global smooth solutions by transport noise of 3d {Navier--Stokes} equations with small hyperviscosity.
\newblock \href{https://arxiv.org/abs/2406.09267}{arxiv.org/abs/2406.09267}, 2024.

\bibitem[Agr24b]{agresti2024AnomalousDissipationInduced}
A.~Agresti.
\newblock On anomalous dissipation induced by transport noise.
\newblock \href{https://arxiv.org/abs/2405.03525}{arxiv.org/abs/2405.03525}, 2024.

\bibitem[BFL24]{butori_luongo_24_magnetic}
F.~Butori., F.~Flandoli, and E.~Luongo.
\newblock On the {It\^o--Stratonovich} diffusion limit for the magnetic field in a {3D} thin domain.
\newblock \href{https://arxiv.org/abs/2401.15701}{arxiv.org/abs/2401.15701}, 2024.

\bibitem[BFLT24]{butori_background_2024}
Federico Butori, Franco Flandoli, Eliseo Luongo, and Yassine Tahraoui.
\newblock Background {Vlasov} equations and {Young} measures for passive scalar and vector advection equations under special stochastic scaling limits.
\newblock \href{https://arxiv.org/abs/2407.10594}{arxiv.org/abs/2407.10594}, July 2024.

\bibitem[BFM16]{brzezniak_existence_nodate}
Zdzisław B., F.~Flandoli, and M.~Maurelli.
\newblock Existence and {uniqueness} for {Stochastic} {2D} {Euler} {Flows} with {Bounded} {Vorticity}.
\newblock {\em Archive for Rational Mechanics and Analysis}, 221(1):107--142, 2016.

\bibitem[BL25]{butori_mean-field_2025}
F.~Butori and E.~Luongo.
\newblock Mean-field magnetohydrodynamics models as scaling limits of stochastic induction equations.
\newblock \href{https://arxiv.org/abs/2406.07206}{arxiv.org/abs/2406.07206}, 2025.

\bibitem[BLP78]{bensoussan1978AsymptoticAnalysisPeriodic}
A.~Bensoussan, J.-L. Lions, and G.~Papanicolaou.
\newblock {\em Asymptotic Analysis for Periodic Structures}.
\newblock Number v. 5 in Studies in Mathematics and Its Applications. {North-Holland Pub. Co. Sole distributors for the U.S.A. and Canada, Elsevier North-Holland}, 1978.

\bibitem[CF96]{chertkov_falkovich_96_anomalous}
M.~Chertkov and G.~Falkovich.
\newblock Anomalous scaling exponents of a white-advected passive scalar.
\newblock {\em Phys. Rev. Lett.}, 76:2706--2709, Apr 1996.

\bibitem[CGH{\etalchar{+}}03]{chaves_gawcedzki_horvai_kupianen_vergassola_03_lagrangian}
M.~Chaves, K.~Gawedzki, P.~Horvai, A.~Kupiainen, and M.~Vergassola.
\newblock Lagrangian dispersion in gaussian self-similar velocity ensembles.
\newblock {\em Journal of Statistical Physics}, 113(5):643--692, 2003.

\bibitem[CM24]{coghi_maurelli_24_kraichnan}
M.~Coghi and M.~Maurelli.
\newblock Existence and uniqueness by {Kraichnan} noise for {2D Euler} equations with unbounded vorticity.
\newblock \href{https://arxiv.org/abs/2308.03216}{arxiv.org/abs/2308.03216}, 2024.

\bibitem[DM93]{maso1993introduction}
G.~Dal~Maso.
\newblock {\em An introduction to {$\Gamma$}-convergence}, volume~8 of {\em Progress in Nonlinear Differential Equations and their Applications}.
\newblock Birkh\"auser Boston, Inc., Boston, MA, 1993.

\bibitem[DP24]{debussche_pappalettera_24_second_order}
Arnaud Debussche and Umberto Pappalettera.
\newblock Second order perturbation theory of two-scale systems in fluid dynamics.
\newblock {\em J. Eur. Math. Soc.}, 2024.

\bibitem[Eva10]{evans_10}
L.~C. Evans.
\newblock {\em Partial {D}ifferential {E}quations}.
\newblock American Mathematical Society, 2010.

\bibitem[FGL21a]{flandoli_galeati_luo_21_delayed}
F.~Flandoli, L.~Galeati, and D.~Luo.
\newblock Delayed blow-up by transport noise.
\newblock {\em Comm. Partial Differential Equations}, 0(46):1--39, 2021.

\bibitem[FGL21b]{flandoli_scaling_2021}
Franco Flandoli, Lucio Galeati, and Dejun Luo.
\newblock Scaling limit of stochastic {2D} {Euler} equations with transport noises to the deterministic {Navier}-{Stokes} equations.
\newblock {\em Journal of Evolution Equations}, 21(1):567--600, March 2021.

\bibitem[FGL22]{flandoli_galeati_luo_22_eddy}
F.~Flandoli, L.~Galeati, and D.~Luo.
\newblock Eddy heat exchange at the boundary under white noise turbulence.
\newblock {\em Philos. Trans. Roy. Soc. A}, 380(2219):Paper No. 20210096, 13, 2022.

\bibitem[FGL24]{flandoli_galeati_luo_21_mixing}
Franco Flandoli, Lucio Galeati, and Dejun Luo.
\newblock Quantitative convergence rates for scaling limit of {SPDE}s with transport noise.
\newblock {\em J. Differential Equations}, 394:237--277, 2024.

\bibitem[FGP10]{flandoli_gubinelli_priola_10_wellposed}
F.~Flandoli, M.~Gubinelli, and E.~Priola.
\newblock Well-posedness of the transport equation by stochastic perturbation.
\newblock {\em Invent. Math.}, 180(1):1--53, 2010.

\bibitem[FL21]{flandoli_luo_21_highmode}
F.~Flandoli and D.~Luo.
\newblock {High mode transport noise improves vorticity blow-up control in 3D Navier–Stokes equations}.
\newblock {\em Probability Theory and Related Fields}, 180(1):309--363, 2021.

\bibitem[FL22]{flandoli_luongo_22_channel}
F.~Flandoli and E.~Luongo.
\newblock Heat diffusion in a channel under white noise modeling of turbulence.
\newblock {\em Math. Eng.}, 4(4):Paper No. 034, 21, 2022.

\bibitem[FL23]{Flandoli_Luongo_Stochastic}
F.~Flandoli and E.~Luongo.
\newblock {\em Stochastic Partial Differential Equations in Fluid Mechanics}, volume Lecture Notes in Mathematics, 2330.
\newblock Springer Nature, 2023.

\bibitem[Fla95]{flandoli_spdes}
F.~Flandoli.
\newblock {\em Regularity theory and stochastic flows for parabolic {SPDE}s}, volume~9 of {\em Stochastics Monographs}.
\newblock Gordon and Breach Science Publishers, Yverdon, 1995.

\bibitem[FLL23]{flandoli_2d_2023}
Franco Flandoli, Dejun Luo, and Eliseo Luongo.
\newblock {2D} {Smagorinsky} type large eddy models as limits of stochastic {PDEs}.
\newblock \href{http://arxiv.org/abs/2302.13614}{arxiv.org/abs/2302.13614}, February 2023.

\bibitem[FMV98]{frisch_mazzino_vergassola_98_intermittency}
U.~Frisch, A.~Mazzino, and M.~Vergassola.
\newblock Intermittency in passive scalar advection.
\newblock {\em Phys. Rev. Lett.}, 80:5532--5535, Jun 1998.

\bibitem[FP94]{fannjiang_papanicolaou_94}
A.~Fannjiang and G.~Papanicolaou.
\newblock Convection enhanced diffusion for periodic flows.
\newblock {\em SIAM Journal on Applied Mathematics}, 54(2):333--408, 1994.

\bibitem[FP22]{flandoli2022additive}
Franco Flandoli and Umberto Pappalettera.
\newblock From additive to transport noise in 2d fluid dynamics.
\newblock {\em Stochastics and Partial Differential Equations: Analysis and Computations}, 10(3):964--1004, 2022.

\bibitem[Gal20]{galeati_20_convergence}
L.~Galeati.
\newblock {On the convergence of stochastic transport equations to a deterministic parabolic one}.
\newblock {\em Stochastics and Partial Differential Equations: Analysis and Computations}, 8(4):833--868, 2020.

\bibitem[GGM24]{galeati_grotto_maurelli_24_anomalous}
L.~Galeati, F.~Grotto, and M.~Maurelli.
\newblock Anomalous regularization in {K}raichnan's passive scalar model.
\newblock \href{https://arxiv.org/abs/2407.16668}{arxiv.org/abs/2407.16668}, 2024.

\bibitem[GLN25]{galeati_leahy_nilssen_25_rough}
L.~Galeati, J.-M. Leahy, and T.~Nilssen.
\newblock On the well-posedness of (nonlinear) rough continuity equations.
\newblock \href{https://arxiv.org/abs/2502.04982}{arxiv.org/abs/2502.04982}, 2025.

\bibitem[Gro07]{grossmann2007numerical}
C.~Grossmann.
\newblock {\em Numerical treatment of partial differential equations}.
\newblock Springer, 2007.

\bibitem[{Kra}68]{kraichnan_68}
R.~H. {Kraichnan}.
\newblock {Small-scale structure of a scalar field convected by turbulence}.
\newblock {\em Physics of Fluids}, 11(5):945--953, May 1968.

\bibitem[Kra94]{kraichnan_94_anomalous}
R.~H. Kraichnan.
\newblock Anomalous scaling of a randomly advected passive scalar.
\newblock {\em Phys. Rev. Lett.}, 72:1016--1019, Feb 1994.

\bibitem[MK99]{majda_kramer_99_turbulence}
A.~J. Majda and P.~R. Kramer.
\newblock Simplified models for turbulent diffusion: theory, numerical modelling, and physical phenomena.
\newblock {\em Phys. Rep.}, 314(4-5):237--574, 1999.

\bibitem[Paz83]{pazy_83_semigroups}
A.~Pazy.
\newblock {\em Semigroups of linear operators and applications to partial differential equations}, volume~44 of {\em Applied Mathematical Sciences}.
\newblock Springer-Verlag, New York, 1983.

\bibitem[RT24]{roveri_well-posedness_2024}
L.~Roveri and F.~Triggiano.
\newblock Well-posedness of rough {2D} {Euler} equation with bounded vorticity.
\newblock \href{https://arxiv.org/abs/2410.24040}{arxiv.org/abs/2410.24040}, October 2024.

\bibitem[She18]{shen_2018_periodic}
Z.~Shen.
\newblock {\em Periodic homogenization of elliptic systems, volume 269 of Operator Theory: Advances and Applications}.
\newblock Advances in Partial Differential Equations. Birkh\"auser/Springer, Cham, 2018.

\bibitem[Sre19]{sreenivasan_19_turbulent}
K.~R. Sreenivasan.
\newblock Turbulent mixing: A perspective.
\newblock {\em Proceedings of the National Academy of Sciences}, 116(37):18175--18183, 2019.

\end{thebibliography}
\bibliographystyle{alpha}
\end{document}